\documentclass[11pt, a4paper]{amsart}
\usepackage{a4wide}
\usepackage{graphicx,enumitem}
\usepackage{subfigure}
\usepackage{units}
\usepackage{color}
\usepackage{algorithm,algorithmic}

\usepackage{ulem}
\usepackage{latexsym}

\usepackage[pdftex,
            pdftitle={Stochastic wave equation on the sphere},
            pdfauthor={D. Cohen and A. Lang},
            bookmarksopen,
            colorlinks,
            linkcolor=black,
            urlcolor=black,
            citecolor=black
]{hyperref}

\usepackage{dsfont}

\allowdisplaybreaks
\usepackage{mathtools}

\newcommand{\IP}{\mathbb{P}}
\newcommand{\R}{\mathbb{R}}
\newcommand{\C}{\mathbb{C}}
\newcommand{\N}{\mathbb{N}}

\newcommand{\IS}{\mathbb{S}}
\newcommand{\IT}{\mathbb{T}}

\newcommand{\ga}{\alpha}
\newcommand{\gb}{\beta}
\newcommand{\gd}{\delta}

\newcommand{\gk}{\kappa}

\newcommand{\gs}{\sigma}

\newcommand{\vp}{\varphi}
\newcommand{\vt}{\vartheta}

\newcommand{\gO}{\Omega}

\newcommand{\cA}{\mathcal{A}}
\newcommand{\cB}{\mathcal{B}}

\newcommand{\cF}{\mathcal{F}}

\newcommand{\cN}{\mathcal{N}}

\newcommand{\cY}{\mathcal{Y}}

\DeclareMathOperator{\E}{\mathbb{E}} %expectation
\let\Re\relax
\DeclareMathOperator{\Re}{Re}
\let\Im\relax
\DeclareMathOperator{\Im}{Im}

\newcommand{\KL}{Karhunen--Lo\`eve }

\newtheorem{lemma}{Lemma}[section]
\newtheorem{proposition}[lemma]{Proposition}

\newtheorem{theorem}[lemma]{Theorem}

\theoremstyle{remark}
\newtheorem{remark}[lemma]{Remark}

\theoremstyle{definition}

\newtheorem{assumption}[lemma]{Assumption}

\definecolor{darkred}{rgb}{.6,0,0}

\newcommand*\diff{\mathop{}\!\mathrm{d}}
\newcommand{\e}{\mathrm{e}}
\newcommand{\ii}{\ensuremath{\mathrm{i}}}

\begin{document}
\title[Approximation of the stochastic wave equation on the sphere]{
Numerical approximation and simulation of the stochastic wave equation on the sphere}

\author[D.~Cohen]{David Cohen} \address[David Cohen]{
\newline Department of Mathematical Sciences
\newline Chalmers University of Technology \& University of Gothenburg
\newline S--412 96 G\"oteborg, Sweden.} \email[]{david.cohen@chalmers.se}

\author[A.~Lang]{Annika Lang} \address[Annika Lang]{
\newline Department of Mathematical Sciences
\newline Chalmers University of Technology \& University of Gothenburg
\newline S--412 96 G\"oteborg, Sweden.} \email[]{annika.lang@chalmers.se}

\thanks{
Acknowledgment. 
%We thank the anonymous referees for the very useful comments that helped to improve the paper.
The work of DC was partially supported by the Swedish Research Council (VR) (project nr.\ 2018-04443). 
The work of AL was partially supported by the Swedish Research Council (VR) (project nr.\ 2020-04170), by the Wallenberg AI, Autonomous Systems and Software Program (WASP) funded by the Knut and Alice Wallenberg Foundation, and by the Chalmers AI Research Centre (CHAIR)}


\begin{abstract}
Solutions to the stochastic wave equation on the unit sphere are approximated by spectral methods. 
Strong, weak, and almost sure convergence rates for the proposed numerical schemes 
are provided and shown to depend only on the smoothness of the driving noise and the initial conditions. Numerical experiments confirm the theoretical rates. 
The developed numerical method is extended to stochastic wave equations 
on higher-dimensional spheres and to the free stochastic Schr\"odinger equation on the unit sphere. 
\end{abstract}

\maketitle
\section{Introduction}%\label{sec:intro}
The recent years have witnessed a strong interest in the theoretical study of (regularity) properties and the simulation
of random fields, especially the ones that are defined by stochastic partial differential equations (SPDEs) 
on Euclidean spaces. This increase in the interest in random fields is due to the huge 
demand from applications as diverse as models for  
the motion of a strand of DNA floating in a fluid~\cite{Dalang2009}, 
climate and weather forecast models~\cite{Hasselmann}, models for the initiation 
and propagation of action potentials in neurons~\cite{1002247}, 
random surface grow models~\cite{PhysRevLett.56.889}, porous media and subsurface flow~\cite{C12}, 
or modeling of fibrosis in atrial tissue~\cite{10.3389/fphys.2018.01052}, for instance. 

Yet, leaving the (by now well understood) Euclidean setting, theoretical results on random fields on Riemannian manifolds 
have just started to pop up in the literature. So far, this research has mostly focused on random fields on the sphere, 
e.\,g., \cite{MR1874654,MR2568294,MP11,MR3404631,MR3769662,MR4091198} and references therein. 
The interest of random fields on spheres is essentially driven by the fact that our planet Earth is 
approximately a sphere. 

One example of an SPDE on the sphere and the main subject of the numerical analysis of this work is the 
\emph{stochastic wave equation}
\begin{equation*}%\label{eq:swe}
\partial_{tt}u(t)-\Delta_{\IS^2}u(t)=\dot{W}(t),
\end{equation*}
driven by an isotropic $Q$-Wiener process. For details on the notation, see below.
Besides the intrinsic mathematical interest, one motivation to study this equation  
comes from~\cite{MR4031900}. This work proposes and analyzes stochastic diffusion models 
for cosmic microwave background (CMB) radiation studies. 
Such models are given by damped wave equations on the sphere with random initial conditions. Since fluctuations in CMB observations may be generated by errors in the CMB map, contamination from the galaxy or distortions in the optics of the telescope \cite{MR4031900}, 
one may be interested in considering a driving noise living on the sphere. 

Unfortunately, to this day, available and well-analyzed algorithms for an efficient simulation of random fields on manifolds 
do not match the current demand from applications. To name a few results from the literature on numerics for SPDEs on manifolds: 
the paper~\cite{MR3404631} proves rates of convergence for a spectral discretization of 
the heat equation on the sphere driven by an additive isotropic Gaussian noise; 
convergence rates of multilevel Monte Carlo finite and spectral element 
discretizations of stationary diffusion equations on the unit sphere with isotropic lognormal diffusion coefficients are considered 
in~\cite{MR3768993}; 
\cite{MR3763911} proposes a simulation method for Gaussian fields defined over spheres cross time; 
a numerical approximation to solutions to random spherical hyperbolic diffusions is analyzed in~\cite{MR4031900}; 
rates of convergence of approximation schemes to solutions to 
fractional SPDEs on the unit sphere are shown in~\cite{123456}; 
the work~\cite{MR3907363} studies a numerical scheme for simulating 
stochastic heat equations on the unit sphere with multiplicative noise; 
in~\cite{MR4059369} multilevel algorithms for the fast simulation of 
nonstationary Gaussian random fields on compact manifolds are analyzed.
We are not aware of any results on numerical approximations of stochastic wave equations on manifolds.

In the present publication, we derive a representation of the infinite-dimensional analytical solution of the stochastic 
wave equation on the sphere driven by an isotropic $Q$-Wiener noise. 
This needs to be numerically approximated in order to be able to efficiently generate sample paths. 
The proposed algorithm is given by the truncation of a series expansion of the analytical solution, see~\eqref{eq:ansatz}.  
We prove strong and almost sure convergence rates of the fully discrete approximation scheme in Proposition~\ref{prop:error}. 
This is then used to show weak convergence results in Proposition~\ref{prop:weak} and Proposition~\ref{prop:weak2}.
It turns out that these rates depend only on the decay of the angular power spectrum of the driving noise and the smoothness of the initial condition while they are independent of the chosen space and time grids. 
We show that depending on the smoothness of test functions, we obtain up to twice the strong order of convergence. 
These results are shown for the stochastic wave equation 
on the unit sphere~$\IS^2$ and then, strong and almost sure convergence results are 
extended to higher-dimensional spheres~$\IS^{d-1}$. 
Finally we obtain similar results for a related equation, namely the \emph{free stochastic Schr\"odinger equation  
on the sphere} driven by an isotropic noise. Observe that the extension of our results to damped and nonlinear problems is not straightforward and needs further analysis. In particular, one would have to deal with additional errors in the space and time discretization. 

A peculiarity in the present approach is that we are able to obtain two equations for the position and velocity component of the stochastic wave equation that can be simulated separately but with respect to two correlated driving noises. Therefore we put some focus on the properties of these correlated random fields and their simulation, see Proposition~\ref{prop:covar_stoch_conv}. With these in place we are able to show convergence of 
the position even when the series expansion of the velocity does not converge. 

The outline of the paper is as follows:
In Section~\ref{sec:setting} we recall definitions
of isotropic Gaussian random fields on~$\IS^2$,  
of the \KL expansion in spherical harmonic functions of 
these fields from~\cite{MP11,MR3404631}, and of Wiener processes on the sphere. 
This then allows us to define the stochastic wave equation on the sphere in Section~\ref{sec:sweS} and analyze its properties based on the semigroup approach. 
In Section~\ref{sec:conv} we approximate solutions to the SPDE  
with spectral methods. In addition, we provide convergence rates 
of these approximations in the $p$-th moment, in the $\IP$-almost sure sense, and in the weak sense. 
Details on the numerical implementation of the studied discretizations are also presented in this section. 
Numerical illustrations of our theoretical findings are given 
in Section~\ref{sec:num}. Although the main focus of the paper is the stochastic wave equation on the unit sphere~$\IS^2$, 
we include two extensions in the last section that can be solved with the developed theory. Namely,  
an extension of the corresponding results to higher-dimensional spheres~$\IS^{d-1} \subset \R^d$ 
and an efficient algorithm for simulating the free stochastic 
Schr\"odinger equation on the sphere with its convergence properties.
\section{Isotropic Gaussian random fields and Wiener processes on the sphere}\label{sec:setting}
We recall some notions and results, mostly from \cite{MR3404631}, 
in order to be able to define SPDEs on the sphere in the next section. 

Throughout, we denote by $(\gO, \cA, (\cF_t)_t, \IP)$ a complete filtered probability space 
and write $\IS^2$ for the unit sphere in~$\R^3$, i.\,e.,
\begin{equation*}
 \IS^2
  = \left\{ x \in \R^3, \|x\|_{\R^3} = 1\right\},
\end{equation*}
where $\|\cdot\|_{\R^3}$ denotes the Euclidean norm. 
Let $(\IS^2,d)$ be the compact metric space with the geodesic metric given by
  \begin{equation*}
   d(x,y)
    = \arccos \left(\langle x,y \rangle_{\R^3}\right)
  \end{equation*}
for all $x,y \in \IS^2$. 
We denote by $\cB(\IS^2)$ the Borel $\gs$-algebra of~$\IS^2$.  

To introduce basis expansions often also called \emph{\KL expansions} of a $Q$-Wiener process on the sphere, we first need to define spherical harmonic functions on~$\IS^2$. 
We recall that the \emph{Legendre polynomials}  
$(P_\ell, \ell \in \N_0)$ are for example 
given by Rodrigues' formula (see, e.\,g., \cite{Szego})
\begin{equation*}
P_\ell(\mu)
= 
2^{-\ell} \frac{1}{\ell!} \, \frac{\partial^\ell}{\partial \mu^\ell} (\mu^2 -1)^\ell
\end{equation*}
for all $\ell \in \N_0$ and $\mu \in [-1,1]$.
These polynomials define the \emph{associated Legendre functions} $(P_{\ell, m}, \ell \in \N_0,m=0,\ldots,\ell)$ by
\begin{equation*}
P_{\ell, m}(\mu)
= (-1)^m (1-\mu^2)^{m/2} \frac{\partial^m}{\partial \mu^m} P_\ell(\mu)
\end{equation*}
for $\ell \in \N_0$, $m = 0,\ldots,\ell$, and $\mu \in [-1,1]$. 
We further introduce
the \emph{surface spherical harmonic functions} 
$\cY = (Y_{\ell, m}, \ell \in \N_0, m=-\ell, \ldots, \ell)$ 
as mappings 
$Y_{\ell, m}\colon [0,\pi] \times [0,2\pi) \to \C$, 
which are given by
\begin{equation*}
Y_{\ell, m}(\vt, \vp)
= \sqrt{\frac{2\ell + 1}{4\pi}\frac{(\ell-m)!}{(\ell+m)!}} P_{\ell, m}(\cos \vt) \e^{\ii m\vp}
\end{equation*}
for $\ell \in \N_0$, $m = 0,\ldots, \ell$, and $(\vt,\vp) \in [0,\pi] \times [0,2\pi)$ and by
\begin{equation*}
Y_{\ell, m}
= (-1)^m \overline{Y_{\ell, -m}}
\end{equation*}
for $\ell \in \N$ and $m=-\ell, \ldots,-1$. It is well-known that the spherical harmonics 
form an orthonormal basis of~$L^2(\IS^2)$, the subspace of real-valued functions in $L^2(\IS^2;\C)$.
In what follows we set for $y \in \IS^2$
\begin{equation*}
Y_{\ell, m}(y)
= Y_{\ell, m}(\vt,\vp),
\end{equation*}
where $y = (\sin \vt \cos \vp, \sin \vt \sin \vp, \cos \vt)$,
i.\,e., we identify (with a slight abuse of notation) 
Cartesian and angular coordinates of the point $y\in \IS^2$.
Furthermore we denote by $\gs$ the \emph{Lebesgue measure on the sphere}
which admits the representation
\begin{equation*}
d\gs(y)
= \sin \vt \, d\vt \, d\vp
\end{equation*}
for $y \in \IS^2$, $y = (\sin \vt \cos \vp, \sin \vt \sin \vp, \cos \vt)$.

The \emph{spherical Laplacian}, also called \emph{Laplace--Beltrami operator}, is given in terms of spherical coordinates similarly to Section~3.4.3 in~\cite{MP11} by
\begin{equation*}
\Delta_{\IS^2}
= (\sin \vt)^{-1} \frac{\partial}{\partial \vt} \left( \sin \vt \, \frac{\partial}{\partial \vt} \right)
+ (\sin \vt)^{-2} \frac{\partial^2}{\partial \vp^2}.
\end{equation*}
It is well-known (see, e.\,g., Theorem~2.13 in~\cite{M98})
that the spherical harmonic functions~$\cY$ 
are the eigenfunctions of~$\Delta_{\IS^2}$ with 
eigenvalues $(-\ell(\ell+1), \ell \in \N_0)$, i.\,e.,
\begin{equation*}
\Delta_{\IS^2} Y_{\ell, m}
= - \ell(\ell+1) Y_{\ell, m}
\end{equation*}
for all $\ell \in \N_0$, $m = -\ell, \ldots,\ell$. 

To characterize the regularity of solutions to SPDEs in what follows, we introduce 
the Sobolev space on~$\IS^2$ for a smoothness index $s\in\R$
$$
H^s(\IS^2)=(\text{Id}-\Delta_{\IS^2})^{-s/2}L^2(\IS^2)
$$
together with its norm
$$
\|f\|_{H^s(\IS^2)}=\|(\text{Id}-\Delta_{\IS^2})^{s/2}f\|_{L^2(\IS^2)}
$$
for some $f\in H^s(\IS^2)$ with $H^0(\IS^2) = L^2(\IS^2)$.

Furthermore, we work on $L^p(\gO;H^s(\IS^2))$ with norm
\begin{equation*}
    \|Z\|_{L^p(\gO;H^s(\IS^2))}
 = \E\left[\|Z\|_{H^s(\IS^2)}^p\right]^{1/p}
\end{equation*}
for finite $p \ge 1$ and are now in place to introduce the following definitions:
 
A $\cA \otimes \cB(\IS^2)$-measurable mapping $Z\colon\Omega \times \IS^2 \rightarrow \R$ is called 
a \emph{real-valued random field} on the unit sphere. 
Such a random field is called \emph{Gaussian} if 
for all $k \in \N$ and $x_1, \ldots, x_k \in \IS^2$,  
the multivariate random variable $(Z(x_1),\ldots,Z(x_k))$ is multivariate 
Gaussian distributed.
Finally, such a random field is called \emph{isotropic} if its 
covariance function only depends on the distance $d(x,y)$, for $x,y\in\IS^2$. 

We recall Theorem~2.3 and Lemma~5.1 in~\cite{MR3404631} on the series expansions of isotropic Gaussian random fields on the sphere. 

\begin{lemma}%\label{lem:simulate_sequence_exp_T}%[Lemma~5.1 in \cite{MR3404631}]\label{lem:simulate_sequence_exp_T}
A centered, isotropic Gaussian random field~$Z$ has a converging \KL expansion 
  \begin{equation*}
   Z = \sum_{\ell=0}^\infty \sum_{m=-\ell}^\ell a_{\ell, m} Y_{\ell, m}
  \end{equation*}
  with $a_{\ell, m} = (Z,Y_{\ell, m})_{L^2(\IS^2)}$ and $A_\ell = \E[a_{\ell, m} \overline{a_{\ell, m}}]$ for all $m=-\ell,\ldots,\ell$, where $(A_\ell,\ell\in\N_0)$ is called the \emph{angular power spectrum} of~$Z$. 
 For $\ell \in \N$, $m = 1,\ldots,\ell$, and $\vt \in [0,\pi]$ set
  \begin{equation*}
   L_{\ell, m}(\vt)
    = \sqrt{\frac{2\ell +1}{4\pi} \frac{(\ell-m)!}{(\ell+m)!}}
	P_{\ell, m}(\cos \vt).
  \end{equation*}
Then for $y= (\sin \vt \cos \vp, \sin \vt \sin \vp, \cos \vt)$
\begin{equation*}
   Z(y)
    = \sum_{\ell=0}^\infty \left(\sqrt{A_\ell} X^1_{\ell, 0} L_{\ell, 0}(\vt) 
	+ \sqrt{2A_\ell} \sum_{m=1}^\ell L_{\ell, m}(\vt) (X^1_{\ell, m} \cos(m\vp) + X^2_{\ell, m} \sin(m \vp))\right)
  \end{equation*}
in law, where $((X^1_{\ell, m}, X^2_{\ell, m}), \ell \in \N_0, m=0,\ldots, \ell)$ is a sequence of independent, real-valued, standard normally distributed random variables and $X^2_{\ell, 0} = 0$ for $\ell \in \N_0$. 
\end{lemma}

In order to simulate solutions to the stochastic wave equation on the sphere, 
we need to approximate the driving noise which can be generated by a sequence of Gaussian random fields. We choose to truncate the above series expansion for an index $\gk \in \N$ and set 
  \begin{equation*}
   Z^\gk(y)
    = \sum_{\ell=0}^\gk \left(\sqrt{A_\ell} X^1_{\ell, 0} L_{\ell, 0}(\vt) 
	+ \sqrt{2A_\ell} \sum_{m=1}^\ell L_{\ell, m}(\vt) (X^1_{\ell, m} \cos(m\vp) + X^2_{\ell, m} \sin(m \vp))\right),
  \end{equation*}
where we recall $y= (\sin \vt \cos \vp, \sin \vt \sin \vp, \cos \vt)$ and $(\vt,\vp) \in [0,\pi]\times [0,2\pi)$.

The above lemma then allows us to present the following results 
on $L^p(\gO;L^2(\IS^2))$ convergence and $\IP$-almost sure convergence of the truncated series 
which are proven in Theorem~5.3 and Corollary~5.4 in~\cite{MR3404631}.

\begin{theorem}%[Theorem~5.3 and Corollary~5.4 in~\cite{MR3404631}]
	\label{thm:iGRF_Lp_conv}
Let the angular power spectrum $(A_\ell, \ell \in \N_0)$ of the centered, isotropic Gaussian random field~$Z$ 
decay algebraically with order $\ga>2$, i.\,e.,
there exist constants $C>0$ and $\ell_0 \in \N$ such that $A_\ell \le C \cdot \ell^{-\ga}$ for all $\ell > \ell_0$.
Then the series of approximate random fields $(Z^\gk, \gk \in \N)$ converges 
to the random field~$Z$ in~$L^p(\gO;L^2(\IS^2))$ for any finite $p\ge 1$, 
and the truncation error is bounded by 
  \begin{equation*}
   \|Z - Z^\gk\|_{L^p(\gO;L^2(\IS^2))}
%   	= \E\left[\|Z - Z^\gk\|_{L^2(\IS^2)}^p\right]^{1/p}
    \le \hat{C}_p \cdot \gk^{-(\ga-2)/2}
  \end{equation*}
for $\gk > \ell_0$, where $\hat{C}_p$ is a constant depending on~$p$, $C$, and~$\ga$.

In addition, $(Z^\gk, \gk \in \N)$ converges $\IP$-almost surely and for all $\gd < (\ga-2)/2$,
the truncation error is asymptotically bounded by
  \begin{equation*}
   \|Z - Z^\gk\|_{L^2(\IS^2)} \le \gk^{-\gd},
    \quad \IP\text{-a.s.}.
  \end{equation*}
\end{theorem}

We follow \cite{MR3404631}, where isotropic Gaussian random fields are connected to $Q$-Wiener processes. There it is shown that an isotropic $Q$-Wiener process $(W(t), t \in \IT)$ on some finite time interval $\IT = [0,T]$ with values in~$L^2(\IS^2)$ can be represented by the expansion 
\begin{align}\label{eqW}
\begin{split}
W(t,y)
& = \sum_{\ell=0}^\infty \sum_{m=-\ell}^\ell a^{\ell, m}(t) Y_{\ell, m}(y)
\\
& = \sum_{\ell=0}^\infty \left(\sqrt{A_\ell} \gb_1^{\ell, 0}(t) Y_{\ell, 0}(y) 
+ \sqrt{2A_\ell} \sum_{m=1}^\ell (\gb_1^{\ell, m}(t) \Re Y_{\ell, m}(y) + \gb_2^{\ell, m}(t) \Im Y_{\ell, m}(y))\right)
\\
& = \sum_{\ell=0}^\infty \left( \sqrt{A_\ell} \gb_1^{\ell, 0}(t) L_{\ell, 0}(\vt) 
+ \sqrt{2A_\ell} \sum_{m=1}^\ell L_{\ell, m}(\vt) (\gb_1^{\ell, m}(t) \cos(m\vp) + \gb_2^{\ell, m}(t) \sin(m \vp))\right),
\end{split}
\end{align}
where $((\gb_1^{\ell, m}, \gb_2^{\ell, m}), \ell \in \N_0, m=0,\ldots, \ell)$ 
is a sequence of independent, real-valued Brownian motions with 
$\gb_2^{\ell, 0} = 0$ for $\ell \in \N_0$ and $t \in \IT$. 
The covariance operator~$Q$ is characterized similarly to the introduction in~\cite{LLS13} by 
\begin{align*}
Q Y_{\ell, m} 
& = A_\ell Y_{\ell, m}
\end{align*}
for $\ell \in \N_0$ and $m=-\ell,\ldots,\ell$, i.\,e., the eigenvalues of~$Q$ 
are given by the angular power spectrum $(A_\ell, \ell \in \N_0)$, 
and the eigenfunctions are the spherical harmonic functions. 

Due to the properties of Brownian motion, the above $Q$-Wiener process can be generated by increments which are isotropic Gaussian random fields with angular power spectrum $(hA_\ell, \ell \in \N_0)$ for a time step size~$h$.

\section{The stochastic wave equation on the sphere}\label{sec:sweS}
With the preparations from the preceding section at hand, we have all necessary tools to 
introduce the main subject of our study.

The \emph{stochastic wave equation on the sphere} is defined as 
\begin{equation}\label{eq:swe}
\partial_{tt}u(t)-\Delta_{\IS^2}u(t)=\dot{W}(t)
\end{equation}
with initial conditions $u(0)=v_1\in L^2(\Omega;L^2(\IS^2))$ and $\partial_{t}u(0)=v_2\in L^2(\Omega;L^2(\IS^2))$, 
where $t \in \IT = [0,T]$, $T < + \infty$. For ease of presentation, we consider the case of non-random initial conditions. The case of random initial conditions follows under appropriate integrability assumptions. The notation $\dot{W}$ stands for the formal 
derivative of the $Q$-Wiener process with series expansion~\eqref{eqW} as introduced in Section~\ref{sec:setting}. 

Denoting the velocity of the solution by $u_2 = \partial_{t} u_1 = \partial_{t} u$, 
one can rewrite \eqref{eq:swe} as 
\begin{align}\label{eq:swe2}
\diff X(t)&=AX(t)\,\diff t+G\,\diff W(t)\nonumber\\
X(0)&=X_0,
\end{align}
where 
\begin{equation*}
 A=\begin{pmatrix}0 & I \\ \Delta_{\IS^2} & 0 \end{pmatrix}, 
 \quad G=\begin{pmatrix}0\\I \end{pmatrix}, 
 \quad X=\begin{pmatrix} u_1\\u_2\end{pmatrix}, 
 \quad X_0=\begin{pmatrix} v_1\\v_2\end{pmatrix}.  
\end{equation*}
Existence of a unique mild solution of the abstract formulation~\eqref{eq:swe2} of the stochastic wave equation 
on the sphere follows from classical results on linear SPDEs, see for instance~\cite{DPZ92}, and this mild solution reads 
\begin{align*}%\label{eq:mild}
X(t)=\e^{tA}X_0+\int_0^t\e^{(t-s)A}G\,\diff W(s).
\end{align*}
Equivalently, the integral formulation of our problem is given by 
\begin{align}\label{eq:int}
\begin{cases}
u_1(t)&\displaystyle=v_1+\int_0^tu_2(s)\,\diff s\\
u_2(t)&\displaystyle=v_2+\int_0^t\Delta_{\IS^2}u_1(s)\,\diff s+W(t).
\end{cases}
\end{align}

Since the spherical harmonic functions~$\cY=(Y_{\ell,m}, \ell\in\N_0, m=-\ell,\ldots,\ell)$ 
form an orthonormal basis of~$L^2(\IS^2)$ and are eigenfunctions of~$\Delta_{\IS^2}$, 
we insert the following ansatz for a series expansion of the exact solution to SPDE~\eqref{eq:swe} 
\begin{align}\label{eq:ansatz}
u_1(t)=\displaystyle\sum_{\ell=0}^\infty\sum_{m=-\ell}^\ell u_1^{\ell,m}(t)Y_{\ell, m}\quad\text{and}\quad 
u_2(t)=\displaystyle\sum_{\ell=0}^\infty\sum_{m=-\ell}^\ell u_2^{\ell,m}(t)Y_{\ell, m}
\end{align}
into equation~\eqref{eq:int} and compare the coefficients in front of $Y_{\ell,m}$ to obtain the following system 
\begin{align*}
u_1^{\ell,m}(t)&=v_1^{\ell,m}+\int_0^t u_2^{\ell,m}(s)\,\diff s\\
u_2^{\ell,m}(t)&=v_2^{\ell,m}-\ell(\ell+1)\int_0^t u_1^{\ell,m}(s)\,\diff s+a^{\ell,m}(t),
\end{align*}
where $v_1^{\ell,m}$, $v_2^{\ell,m}$, resp. $a^{\ell,m}$  
are the coefficients of the expansions of the initial values $v_1$ and $v_2$, resp.\ weighted Brownian motions in the expansion of the noise~\eqref{eqW}. 

Writing the evolution of the initial values in the above linear harmonic oscillators 
with rotation matrices and using the variation of constants formula, 
one derives the following system for the coefficients of the expansions of the solution
\begin{align}\label{eq:sys}
\begin{cases}
u_1^{\ell,m}(t)&\displaystyle=\cos(t(\ell(\ell+1))^{1/2})v_1^{\ell,m}+(\ell(\ell+1))^{-1/2}\sin(t(\ell(\ell+1))^{1/2})v_2^{\ell,m}+\hat W_1^{\ell,m}(t)\\
u_2^{\ell,m}(t)&\displaystyle=-(\ell(\ell+1))^{1/2}\sin(t(\ell(\ell+1))^{1/2})v_1^{\ell,m}+\cos(t(\ell(\ell+1))^{1/2})v_2^{\ell,m}+\hat W_2^{\ell,m}(t),
\end{cases}
\end{align}
where 
\begin{equation*}
\hat W^{\ell,m}(t)
=
\begin{pmatrix}
\hat W_1^{\ell,m}(t)\\
\hat W_2^{\ell,m}(t)
\end{pmatrix}
=\displaystyle\int_0^t R^\ell(t-s) \,\diff a^{\ell,m}(s)
\end{equation*}
with
\begin{equation*}
R^\ell (t)
= \begin{pmatrix}
R^\ell_1(t) \\ R^\ell_2(t)
\end{pmatrix}
= \begin{pmatrix}
(\ell(\ell+1))^{-1/2}\sin(t(\ell(\ell+1))^{1/2})\\
\cos(t(\ell(\ell+1))^{1/2})
\end{pmatrix}
\end{equation*}
for $\ell \neq 0$ and
\begin{equation*}
\hat W^{0,0}(t)
=
\begin{pmatrix}
\hat W_1^{0,0}(t)\\
\hat W_2^{0,0}(t)
\end{pmatrix}
= \begin{pmatrix}
\displaystyle\int_0^t a^{0,0}(s) \, \diff s\\
a^{0,0}(t)
\end{pmatrix}.
\end{equation*}
We now characterize the above stochastic convolutions $\hat W_i^{\ell,m}$ for $i=1,2$.
\begin{proposition} \label{prop:covar_stoch_conv}
The stochastic convolution $\hat W(t)$ is Gaussian with mean zero and expansion
    \begin{align*}
	\hat W(t,y)
	& = \sum_{\ell=0}^\infty \sum_{m=-\ell}^\ell W^{\ell, m}(t) Y_{\ell, m}(y)
	\\
	& = \sum_{\ell=0}^\infty \left(\sqrt{A_\ell} \hat \gb_1^{\ell, 0}(t) Y_{\ell, 0}(y) 
	+ \sqrt{2A_\ell} \sum_{m=1}^\ell (\hat \gb_1^{\ell, m}(t) \Re Y_{\ell, m}(y) + \hat \gb_2^{\ell, m}(t) \Im Y_{\ell, m}(y))\right)
	\\
	& = \sum_{\ell=0}^\infty \left( \sqrt{A_\ell} \hat \gb_1^{\ell, 0}(t) L_{\ell, 0}(\vt) 
	+ \sqrt{2A_\ell} \sum_{m=1}^\ell L_{\ell, m}(\vt) (\hat \gb_1^{\ell, m}(t) \cos(m\vp) + \hat \gb_2^{\ell, m}(t) \sin(m \vp))\right),
	\end{align*}
	where equality is in distribution.
	
	The processes $(\hat \gb^{\ell,m}_i(t), i=1,2, \ell \in \N_0, m=-\ell,\ldots, \ell)$ are given by
	\begin{equation*}
	 \hat \gb^{\ell,m}_i(t)
	 	= \begin{pmatrix}
	 	 \gb_{i,1}^{\ell,m}(t)\\ \gb_{i,2}^{\ell,m}(t)
	 	\end{pmatrix}
	 	= D_\ell(t) X^{\ell,m}_i
	\end{equation*}
	for a sequence $(X^{\ell,m}_i = (X_{i,1}^{\ell,m},X_{i,2}^{\ell,m})^T, i=1,2, \ell \in \N_0, m = -\ell, \ldots, \ell)$ of independent, identically distributed random variables with $X^{\ell,m}_{i,j} \sim \cN(0,1)$.
	The term $D_\ell(t)$ denotes the Cholesky decomposition of the covariance matrix $C_\ell(t)$ of $\hat W^{\ell,m}(t)$. 
	More specifically, $D_\ell$ satisfies
	\begin{equation*}
	 D_\ell(t) D_\ell(t)^T = C_\ell(t)
	\end{equation*}
	with
	\begin{equation*}
	 C_\ell(t)
	 	= 
	 	\begin{pmatrix}
	 	\frac{2(\ell(\ell+1))^{1/2}t-\sin(2(\ell(\ell+1))^{1/2}t)}{4(\ell(\ell+1))^{3/2}} & \frac{\sin((\ell(\ell+1))^{1/2}t)^2}{2(\ell(\ell+1))}\\
	 	\frac{\sin((\ell(\ell+1))^{1/2}t)^2}{2(\ell(\ell+1))} & \frac{2(\ell(\ell+1))^{1/2}t+\sin(2(\ell(\ell+1))^{1/2}t)}{4(\ell(\ell+1))^{1/2}}
	 	\end{pmatrix}
	\end{equation*}
	for $\ell \neq 0$ and
	\begin{equation*}
	 C_0(t)
	 	= 
	 	\begin{pmatrix}
	 	t^3/3 & t^2/2\\
	 	t^2/2 & t
	 	\end{pmatrix}.
	\end{equation*}
\end{proposition}

\begin{proof}
We observe first that $\hat W(t)$ satisfies	by~\eqref{eqW}
\begin{align*}
\hat W(t)
 &= \sum_{\ell = 0}^\infty \sum_{m=-\ell}^\ell \hat W^{\ell,m}(t) Y_{\ell,m}\\
& =  \hat W^{0,0}(t) Y_{0,0} + \sum_{\ell = 1}^\infty \sum_{m=-\ell}^\ell \int_0^t R^\ell(t-s) \,\diff a^{\ell,m}(s) Y_{\ell,m}\\
& = \hat W^{0,0}(t) Y_{0,0} + \sum_{\ell = 1}^\infty \sqrt{A_\ell}
\left[
\int_0^t R^\ell(t-s) \, \diff \gb_1^{\ell,0}(s) Y_{\ell,0}\right.\\
& \left.\hspace{5em} + \sqrt{2} \sum_{m=1}^\ell
\left(
\int_0^t R^\ell(t-s) \, \diff \gb_1^{\ell,m}(s) \Re Y_{\ell,m}
+ \int_0^t R^\ell(t-s) \, \diff \gb_2^{\ell,m}(s) \Im Y_{\ell,m}
\right)
\right]
\end{align*}
with independent Brownian motions $(\gb_1^{\ell,m}, \ell \in \N_0, m=0,\ldots,\ell)$ 
and $(\gb_2^{\ell,m}, \ell \in \N, m=1,\ldots,\ell)$. 
Since all Brownian motions are centered and independent, 
it is sufficient to compute the following covariances which are given by
\begin{align*}
  C_0(t)
 	&=A_0^{-1} 
	\begin{pmatrix}
		\E[\hat W^{0,0}_1(t)^2] & \E[\hat W_1^{0,0}(t)\hat W_2^{0,0}(t)]\\
		\E[\hat W_1^{0,0}(t)\hat W_2^{0,0}(t)] & \E[\hat W_2^{0,0}(t)^2]
	\end{pmatrix}\\
	&=A_0^{-1} \begin{pmatrix}
	\E[(\int_0^t a^{0,0}(s) \, \diff s)^2] 
		& \E[ \int_0^t a^{0,0}(s) a^{0,0}(t) \, \diff s]\\
		\E[ \int_0^t a^{0,0}(s) a^{0,0}(t) \, \diff s]
		& \E[(a^{0,0}(t))^2]
	\end{pmatrix}
	=  
	\begin{pmatrix}
		t^3/3 & t^2/2\\
		t^2/2 & t
	\end{pmatrix}
\end{align*}
for $\ell = 0$ and else for $i=1,2$
\begin{align*}
 C_\ell(t)
 	& = 
 	\begin{pmatrix}
 	 \E[(\text{Int}_1)^2]
 	 	& \E[\text{Int}_1\text{Int}_2]\\
 	 \E[\text{Int}_1 \text{Int}_2]
 	 	& \E[(\text{Int}_2)^2]
 	\end{pmatrix}\\
 	& = 
 	\begin{pmatrix}
 	\int_0^t R^\ell_1(t-s)^2 \, \diff s
 	& \int_0^t R^\ell_1(t-s) R^\ell_2(t-s) \, \diff s\\
 	\int_0^t R^\ell_1(t-s) R^\ell_2(t-s) \, \diff s
 	& \int_0^t R^\ell_2(t-s)^2 \, \diff s
 	\end{pmatrix} \\
 	& = 
	\begin{pmatrix}
	\frac{2(\ell(\ell+1))^{1/2}t-\sin(2(\ell(\ell+1))^{1/2}t)}{4(\ell(\ell+1))^{3/2}} 
		& \frac{\sin((\ell(\ell+1))^{1/2}t)^2}{2(\ell(\ell+1))}\\
	\frac{\sin((\ell(\ell+1))^{1/2}t)^2}{2(\ell(\ell+1))} 
		& \frac{2(\ell(\ell+1))^{1/2}t+\sin(2(\ell(\ell+1))^{1/2}t)}{4(\ell(\ell+1))^{1/2}}
	\end{pmatrix},
\end{align*}
where we have set $\text{Int}_1=\int_0^t R^\ell_1(t-s) \, \diff \gb_i^{\ell,m}(s)$ and 
$\text{Int}_2=\int_0^t R^\ell_2(t-s) \, \diff \gb_i^{\ell,m}(s)$.

Setting $D_\ell(t)$ the Cholesky decomposition of the above covariance matrices satisfying
\begin{equation*}
 D_\ell(t)^T D_\ell(t)
	= C_\ell(t)
\end{equation*}
we obtain for $\ell \neq 0$ and $i=1,2$ that
\begin{equation*}
 D_\ell(t)X^{\ell,m}_i
 	= \int_0^t R^\ell(t-s) \, \diff \gb_i^{\ell,m}(s)
\end{equation*}
in distribution and similarly for $\ell = 0$
\begin{equation*}
D_\ell(t)X^{\ell,m}_1
= \hat W^{0,0}(t)
\end{equation*}
with $(X^{\ell,m}_i = (X_{i,1}^{\ell,m},X_{i,2}^{\ell,m})^T, i=1,2, \ell \in \N_0, m = -\ell, \ldots, \ell)$ 
independent and identically distributed standard normally distributed random variables.
This concludes the proof.
\end{proof}
	
\begin{remark}
	Since we are interested in the simulation of sample paths of solutions to~\eqref{eq:swe2}, we need to generate increments of $\hat W^{\ell,m}(t)$. Therefore it is important to observe that
	\begin{equation*}
	 \hat \gb^{\ell,m}_i(t) - \hat \gb^{\ell,m}_i(s)
	 	= D_\ell(t-s) X^{\ell,m}_i
	\end{equation*}
	in distribution for $s<t$. In this way we can generate sample paths of $\hat W(t)$ by sums of independent Gaussian increments.
	
For completeness we also remark that the Cholesky decomposition~$D_\ell(t)$ can be computed explicitly and is given by
\begin{equation}\label{eq:expl_Cholesky}
 D_\ell(t)
 	=
 	\begin{pmatrix}
 		 d_{1,1}
 		& d_{1,2}\\
 		0 
 		& d_{2,2}
 	\end{pmatrix}
\end{equation}
with
\begin{align*}
 d_{1,1} & = \frac{(2(\ell(\ell+1))^{1/2}t-\sin(2(\ell(\ell+1))^{1/2}t))^{1/2}}{2(\ell(\ell+1))^{3/4}}\\
 d_{1,2} & = \frac{\sin((\ell(\ell+1))^{1/2}t)^2}{(\ell(\ell+1))^{1/4}(2(\ell(\ell+1))^{1/2}t-\sin(2(\ell(\ell+1))^{1/2}t))^{1/2}}\\
 d_{2,2} & = \left(\frac{4(\ell(\ell+1))t^2 - \sin(2(\ell(\ell+1))^{1/2}t)^2 - 4\sin((\ell(\ell+1))^{1/2}t)^4}{4(\ell(\ell+1))^{1/2}(2(\ell(\ell+1))^{1/2}t-\sin(2(\ell(\ell+1))^{1/2}t))} \right)^{1/2}
\end{align*}
for $\ell \neq 0$ and
\begin{equation*}
 D_0(t)
 	= t^{1/2}
	\begin{pmatrix}
		t/\sqrt{3} & \sqrt{3}/2\\
		0 & 1/2
	\end{pmatrix}. 	
\end{equation*}
\end{remark}

We close this section by showing regularity estimates for the solution of~\eqref{eq:swe} that depend 
on the regularity of the initial conditions and the driving noise.
These properties allow to obtain optimal weak convergence rates in Section~\ref{sec:conv}.

\begin{proposition}\label{prop:regularity}
	Denote by $X=(u_1,u_2)$ the solution to the stochastic wave equation~\eqref{eq:swe2} with initial value $(v_1,v_2)$. 
	Assume that there exist $\ell_0 \in \N$, $\ga > 2$, and a constant~$C>0$ 
	such that the angular power spectrum of the driving noise $(A_\ell, \ell \in \N_0)$ satisfies 
	$A_\ell \le C \cdot \ell^{-\ga}$ for all $\ell > \ell_0$.
	Then, for all $t \in [0,T]$, $m \in \N$, and $s < \ga/2$ with $v_1 \in H^s(\IS^2)$ and $v_2 \in H^{s-1}(\IS^2)$, $u_1(t) \in L^{2m}(\gO;H^s(\IS^2))$, i.\,e., there exists a constant~$M$ such that
	\begin{equation*}
		\|u_1(t)\|_{L^{2m}(\gO;H^s(\IS^2))}
			\le M (1 + \|v_1\|_{H^s(\IS^2)} + \|v_2\|_{H^{s-1}(\IS^2)})
			< + \infty.
	\end{equation*}
	And for all $t \in [0,T]$, $m \in \N$, and $s < \ga/2 - 1$ with $v_1 \in H^{s+1}(\IS^2)$ and $v_2 \in H^s(\IS^2)$, $u_2(t) \in L^{2m}(\gO;H^s(\IS^2))$, i.\,e., there exists a constant~$M$ such that
	\begin{equation*}
	\|u_2(t)\|_{L^{2m}(\gO;H^s(\IS^2))}
	\le M (1 + \|v_1\|_{H^{s+1}(\IS^2)} + \|v_2\|_{H^s(\IS^2)})
	< + \infty.
	\end{equation*}
	
\end{proposition}

\begin{proof}
	Let us first observe that
	\begin{align*}
	 & \|u_1(t)\|_{L^{2m}(\gO;H^s(\IS^2))}\\
		& \hspace*{3em} \leq \left\|\sum_{\ell=0}^\infty\sum_{m=-\ell}^\ell\left( R^\ell_2(t) v_1^{\ell,m}
	 	+R^\ell_1(t) v_2^{\ell,m}\right)Y_{\ell,m}\right\|_{L^{2m}(\gO;H^s(\IS^2))}
	 	+\left\|\hat{W}(t)\right\|_{L^{2m}(\gO;H^s(\IS^2))}.\\
	\end{align*}
	
	The first term with respect to the initial conditions satisfies
	\begin{align*}
	 & \left\|\sum_{\ell=1}^\infty\sum_{m=-\ell}^\ell\left( R^\ell_2(t) v_1^{\ell,m}
	 +R^\ell_1(t) v_2^{\ell,m}\right)Y_{\ell,m}\right\|_{H^s(\IS^2)}^2\\
	 & \qquad \le 2 \sum_{\ell=1}^\infty\sum_{m=-\ell}^\ell \left(
	 	(1+\ell(\ell+1))^s |v_1^{\ell,m}|^2 + (1+\ell(\ell+1))^s (\ell(\ell+1))^{-1} |v_2^{\ell,m}|^2
	 \right)\\
	 & \qquad \le C (\|v_1\|_{H^s(\IS^2)}^2 + \|v_2\|_{H^{s-1}(\IS^2)}^2).
	\end{align*} 

	Given the angular power spectrum of~$\hat{W}(t)$ in Proposition~\ref{prop:covar_stoch_conv}, it follows 
	for the second moment, i.\,e. $m=1$, that
	\begin{align*}
		\left\|\hat{W}(t)\right\|_{L^{2}(\gO;H^s(\IS^2))}^2
			& = \left\|(\text{Id}-\Delta_{\IS^2})^{s/2} \hat{W}(t)\right\|_{L^{2}(\gO;L^2(\IS^2))}^2\\
			& = \sum_{\ell = 0}^\infty (2\ell+1)(1+\ell(\ell+1))^s \frac{2(\ell(\ell+1))^{1/2}t-\sin(2(\ell(\ell+1))^{1/2}t)}{4(\ell(\ell+1))^{3/2}}A_\ell,
	\end{align*}
which converges for $\ga > 2s$ since the elements of the sum behave like $\ell^{2s-\ga-1}$. 
By Fernique's theorem~\cite{F70}, this convergence implies that 
the norm is finite for all~$m$ and arbitrary moment bounds can for example be obtained 
by the Burkholder--Davis--Gundy inequality.

Similar computations for $u_2$ conclude the proof.
\end{proof}
\section{Convergence analysis}\label{sec:conv}
In this section, we numerically solve the wave equation on the sphere driven by additive $Q$-Wiener noise  
with spectral methods. We approximate the
solution by truncation of the derived spectral representation and 
show convergence rates in $p$-th moment, $\IP$-almost surely, and in the weak sense. 

An efficient simulation of numerical approximations to 
solutions to the stochastic wave equation on the sphere~\eqref{eq:swe} 
is then obtained via Algorithm~\ref{algo}.
\begin{algorithm}
\caption{Simulations of paths of the solution to~\eqref{eq:swe}}
\label{algo}
\begin{algorithmic}[1]
\STATE Fix a truncation index $\gk\in\N$.
\STATE Compute a discrete time grid $0=t_0<t_1<\ldots<t_n=T$, $n\in\N$, with time step $h$.
\STATE Compute the covariance matrix~$C_\ell(h)$ of the stochastic integrals 
$\hat W_1^{\ell,m}(h)$ and $\hat W_2^{\ell,m}(h)$ in Proposition~\ref{prop:covar_stoch_conv}. 
\STATE Perform a Cholesky decomposition of $C_\ell(h)=D_\ell(h)^TD_\ell(h)$ or use 
the explicit formula~\eqref{eq:expl_Cholesky}.
\STATE Use $D_\ell(h)$ to generate noise increments 
$$
\hat W^{\ell,m}(h)
=\begin{pmatrix}
\hat W_1^{\ell,m}(h)\\
\hat W_2^{\ell,m}(h)
\end{pmatrix}
=D_\ell
\begin{pmatrix}
\beta_1^{\ell,m}(h)\\
\beta_2^{\ell,m}(h)
\end{pmatrix}
\sqrt{A_\ell}.
$$
\STATE Compute $u_1^{\ell,m}(t_j+h)$ and $u_2^{\ell,m}(t_j+h)$ recursively using \eqref{eq:sys}
\begin{align*}
 u^{\ell,m}(t_{j+1})
 	= 
 	\begin{pmatrix}
 		R_2^\ell(h) \\ - \ell(\ell+1) R_1^\ell(h)
 	\end{pmatrix}
 	u_1^{\ell,m}(t_j)
 	+ R^\ell(h) u_2^{\ell,m}(t_j)
 	+ \hat W^{\ell,m}(h).
\end{align*}
% \eqref{eq:sys}. 
\STATE Truncate the ansatz \eqref{eq:ansatz} at the fixed positive integer $\gk$ to get the numerical approximations
\begin{align}\label{eq:ansatzK}
u_1^\kappa(t_j)=\displaystyle\sum_{\ell=0}^\gk\sum_{m=-\ell}^\ell u_1^{\ell,m}(t_j)Y_{\ell, m}\quad\text{and}\quad 
u_2^\kappa(t_j)=\displaystyle\sum_{\ell=0}^\gk\sum_{m=-\ell}^\ell u_2^{\ell,m}(t_j)Y_{\ell, m}.
\end{align}
\end{algorithmic}
\end{algorithm}
The strong errors of this truncation procedure are given in the following proposition.

\begin{proposition}\label{prop:error} 
Let $t \in \IT$ and $0=t_0 < \cdots < t_n = t$ be a discrete time partition 
for $n \in \N$, which yields a recursive representation of the solution~$X=(u_1,u_2)$ of 
the stochastic wave equation on the sphere~\eqref{eq:swe2} given by \eqref{eq:ansatz}. 
Assume that the initial values satisfy $v_1\in H^\beta(\IS^2)$ and $v_2\in H^\gamma(\IS^2)$. 
Furthermore, assume that there exist $\ell_0 \in \N$, $\ga > 2$, and a constant~$C>0$ 
such that the angular power spectrum of the driving noise $(A_\ell, \ell \in \N_0)$ satisfies 
$A_\ell \le C \cdot \ell^{-\ga}$ for all $\ell > \ell_0$. 
Then, the error of the approximate solution $X^\gk=(u_1^\gk,u_2^\gk)$, given by \eqref{eq:ansatzK}, 
is bounded uniformly in time and independently of the time discretization by
\begin{align*}
\|u_1(t) - u_1^\gk(t)\|_{L^p(\gO;L^2(\IS^2))} &\le \hat{C}_p \cdot \left( \gk^{-\ga/2}+\gk^{-\beta}\|v_1\|_{H^\beta(\IS^2)}+
\gk^{-(\gamma+1)}\|v_2\|_{H^\gamma(\IS^2)} \right)\\
\|u_2(t) - u_2^\gk(t)\|_{L^p(\gO;L^2(\IS^2))} &\le \hat{C}_p \cdot \left( \gk^{-(\ga/2-1)}+\gk^{-(\beta-1)}\|v_1\|_{H^\beta(\IS^2)}+
\gk^{-\gamma}\|v_2\|_{H^\gamma(\IS^2)} \right)
\end{align*}
for all $p\geq1$ and $\gk > \ell_0$, where $\hat{C}_p$ is a constant that may depend on 
$p$, $C$, $T$, and $\ga$.

On top of that, the error of the approximate solution $X^\gk$ is bounded uniformly in time, 
independently of the time discretization, and asymptotically in~$\gk$ by
\begin{align*}
\|u_1(t) - u_1^\gk(t)\|_{L^2(\IS^2)} &\le \gk^{-\delta}, \quad \IP\text{-a.s.} \\
\|u_2(t) - u_2^\gk(t)\|_{L^2(\IS^2)} &\le \gk^{-(\delta-1)}, \quad \IP\text{-a.s.} 
\end{align*}
for all $\delta < \min(\ga/2,\beta,\gamma+1)$.
\end{proposition}
\begin{remark}\label{rem1}
We remark that it is not necessary that the angular power spectrum 
$(A_\ell, \ell \in \N_0)$ of the $Q$-Wiener process decays with rate
$\ell^{-\ga}$ for $\ga >2$ 
but that it is sufficient to assume that $\ga > 0$ to show convergence in the first component, 
see the numerical experiment in Section~\ref{sec:num}. I.\,e., we do not require that $Q$ is a trace class operator for convergence in the first component.
\end{remark}
\begin{proof}[Proof of Proposition~\ref{prop:error}]
Let us first consider the convergence in $p$-th moment of the first component of the solution for $p\geq1$. 
By definition of $u_1$ and $u_1^\gk$ in~\eqref{eq:ansatz} and~\eqref{eq:ansatzK}, one obtains
\begin{align*}
&\left\|u_1(t)-u_1^\gk(t)\right\|_{L^p(\gO;L^2(\IS^2))}\\
%	= \E\left[\|u_1(t)-u_1^\gk(t)\|_{L^2(\IS^2)}^p\right]^{1/p}
	&\qquad = \left\|\sum_{\ell=\gk+1}^\infty\sum_{m=-\ell}^\ell u_1^{\ell,m}(t)Y_{\ell,m}\right\|_{L^p(\gO;L^2(\IS^2))}\\
	&\qquad \leq \left\|\sum_{\ell=\gk+1}^\infty\sum_{m=-\ell}^\ell\left( R^\ell_2(t) v_1^{\ell,m}
		+R^\ell_1(t) v_2^{\ell,m}\right)Y_{\ell,m}\right\|_{L^p(\gO;L^2(\IS^2))}\\
	&\qquad \hspace*{3em}+\left\|\sum_{\ell=\gk+1}^\infty\sum_{m=-\ell}^\ell\int_0^t
R^\ell_1(t-s)\,\diff a^{\ell,m}(s)Y_{\ell,m}\right\|_{L^p(\gO;L^2(\IS^2))}\\
	&\qquad\leq \|v_1 - v_1^\gk\|_{L^2(\IS^2)} + \left(\sum_{\ell=\gk+1}^\infty\sum_{m=-\ell}^\ell(\ell(\ell+1))^{-1}|v_2^{\ell,m}|^2\right)^{1/2}
		+ \|Z-Z^\gk\|_{L^p(\gO;L^2(\IS^2))}, 
\end{align*}
where $Z$ denotes an isotropic Gaussian random field with angular power spectrum 
$$
\widetilde A_\ell=\frac{2(\ell(\ell+1))^{1/2}t-\sin(2(\ell(\ell+1))^{1/2}t)}{4(\ell(\ell+1))^{3/2}}A_\ell
$$
and $Z^\gk$ its truncation.
Observe that 
\begin{equation*}
 \widetilde A_\ell
 \leq C\left(\frac{t}{2\ell(\ell+1)}+\frac{|\sin(2(\ell(\ell+1))^{1/2}t)|}{4(\ell(\ell+1))^{3/2}}\right)A_\ell
 \leq C T \ell^{-2}A_\ell
 \leq C T \ell^{-(\alpha+2)}
\end{equation*}
for large enough index $\ell > \ell_0 \ge 0$ 
using the assumption on the angular power spectrum of the noise $A_\ell$. 
The assumptions on the initial values and the fact that the spherical harmonic functions are orthonormal 
provide us with the estimate 
\begin{align*}
\left\|v_1 - v_1^\gk\right\|_{L^2(\IS^2)}
	&=\left\| \sum_{\ell=\gk+1}^\infty\sum_{m=-\ell}^\ell v_1^{\ell,m}Y_{\ell,m}\right\|_{L^2(\IS^2)} \\
	&=\left\| \sum_{\ell=\gk+1}^\infty\sum_{m=-\ell}^\ell v_1^{\ell,m}(\text{Id}-\Delta_{\IS^2})^{\beta/2}
(\text{Id}-\Delta_{\IS^2})^{-\beta/2}Y_{\ell,m}\right\|_{L^2(\IS^2)}\\
	&=\left\| \sum_{\ell=\gk+1}^\infty\sum_{m=-\ell}^\ell v_1^{\ell,m}(1+\ell(\ell+1))^{-\beta/2}(\text{Id}-\Delta_{\IS^2})^{\beta/2}
Y_{\ell,m}\right\|_{L^2(\IS^2)}\\
	&\leq C \gk^{-\gb}\left\| v_1\right\|_{H^\beta(\IS^2)}
\end{align*}
and similarly for the second component of the initial value. 

Collecting all the estimates above and using Theorem~\ref{thm:iGRF_Lp_conv} we obtain the desired bound
$$
\|u_1(t)-u_1^\gk(t)\|_{L^p(\gO;L^2(\IS^2))}
	\leq \hat C_p\cdot\left(\gk^{-\ga/2}+\gk^{-\beta}\|v_1\|_{H^\beta(\IS^2)}
		+ \gk^{-(\gamma+1)}\|v_2\|_{H^\gamma(\IS^2)}\right).
$$
The corresponding estimate for the second component is done in a similar way and left to the reader. Observe that 
the rate of convergence decays by one due to the factor $(\ell(\ell+1))^{1/2}$ in the first 
term of~\eqref{eq:sys}.

We continue with the rate of the almost sure convergence in the first component of the solution. 
Let $\delta < \min(\ga/2,\beta,\gamma+1)$. The above strong $L^p$ error estimate combined with Chebyshev's 
inequality provide us with 
 \begin{align*}
   \IP\left(\|u_1(t)-u_1^\gk(t)\|_{L^2(\IS^2)} \ge \gk^{-\delta}\right)
    &\le \gk^{\delta p} \E\left[\|u_1(t)-u_1^\gk(t)\|_{L^2(\IS^2)}^p\right]\\
    &\le \gk^{\delta p}\hat C_p^p\left(\gk^{-\ga/2}+\gk^{-\beta}\|v_1\|_{H^\beta(\IS^2)}+
\gk^{-(\gamma+1)}\|v_2\|_{H^\gamma(\IS^2)}\right)^p.
  \end{align*}
 For all $p > \max( (\ga/2-\delta)^{-1}, (\beta-\delta)^{-1}, (\gamma+1-\delta)^{-1} )$, the series
  \begin{equation*}
   \sum_{\gk=1}^\infty \gk^{(\delta-\min(\ga/2,\beta,\gamma+1))p}
    < + \infty
  \end{equation*}
 converges which implies the claim by the Borel--Cantelli lemma. Almost sure convergence of the second component is shown in a similar way which concludes the proof. 
\end{proof}
Using Proposition~\ref{prop:error}, we continue with bounding weak errors of the mean and second moment in a first step. 
We observe that the weak error for the mean is the error to the corresponding deterministic wave equation on~$\IS^2$ and 
that the error for the second moment satisfies the rule of thumb that the weak convergence rate is twice the strong 
convergence rate.
\begin{proposition}\label{prop:weak}
Let $t \in \IT$ and $0=t_0 < \cdots < t_n = t$ be a discrete time partition 
for $n \in \N$ which yields a recursive representation of the solution~$X=(u_1,u_2)$ of 
the stochastic wave equation on the sphere~\eqref{eq:swe2} given by~\eqref{eq:ansatz}. 
Assume that the initial values satisfy $v_1\in H^\beta(\IS^2)$ and $v_2\in H^\gamma(\IS^2)$. 
Furthermore, assume that there exist $\ell_0 \in \N$, $\ga > 2$, and a constant~$C>0$ 
such that the angular power spectrum of the driving noise $(A_\ell, \ell \in \N_0)$ satisfies 
$A_\ell \le C \cdot \ell^{-\ga}$ for all $\ell > \ell_0$. 

Then, the errors in mean of the approximate solution $X^\gk=(u_1^\gk,u_2^\gk)$, given by~\eqref{eq:ansatzK}, 
are bounded uniformly in time and independently of the time discretization by
\begin{align*}
\left\|\E\left[u_1(t) - u_1^\gk(t)\right]\right\|_{L^2(\IS^2)} 
	&\le \hat{C} \cdot \left( \gk^{-\beta}\|v_1\|_{H^\beta(\IS^2)}
		+ \gk^{-(\gamma+1)}\|v_2\|_{H^\gamma(\IS^2)} \right)\\
\left\|\E\left[u_2(t) - u_2^\gk(t)\right]\right\|_{L^2(\IS^2)} 
	&\le \hat{C} \cdot \left( \gk^{-(\beta-1)}\|v_1\|_{H^\beta(\IS^2)}
		+\gk^{-\gamma}\|v_2\|_{H^\gamma(\IS^2)} \right)
\end{align*}
for all $\gk > \ell_0$, where $\hat{C}$ is a constant that may depend on $C$, $T$, and $\ga$.

Furthermore, the errors of the second moment are bounded by 
\begin{align*}
\left| \E\left[ \|u_1(t)\|^2_{L^2(\IS^2)} - \|u_1^\gk(t)\|^2_{L^2(\IS^2)}\right] \right| 
	&\le \hat{C} \cdot \left( \gk^{-\ga}+\gk^{-2\beta}\|v_1\|_{H^\beta(\IS^2)}
			+ \gk^{-2(\gamma+1)}\|v_2\|_{H^\gamma(\IS^2)} \right)\\
\left| \E\left[ \|u_2(t)\|^2_{L^2(\IS^2)} - \|u_2^\gk(t)\|^2_{L^2(\IS^2)}\right] \right|
	&\le \hat{C} \cdot \left( \gk^{-(\ga-2)}+\gk^{-2(\beta-1)}\|v_1\|_{H^\beta(\IS^2)}
		+ \gk^{-2\gamma}\|v_2\|_{H^\gamma(\IS^2)} \right)
\end{align*}
for all $\gk > \ell_0$, where $\hat{C}$ is a constant that may depend on $C$, $T$, and $\ga$.
\end{proposition}

\begin{proof}
The definition of~$X$ and its approximation~$X^\gk$ yield
$$
\E\left[ u_j(t)-u_j^\gk(t) \right]=\sum_{\ell=0}^\infty\sum_{m=-\ell}^\ell\E\left[u_j^{\ell,m}(t)\right]Y_{\ell,m}
$$
for $j=1,2$. Next, using~\eqref{eq:sys} and the properties of $\hat W_1^{\ell,m}(t)$ from Proposition~\ref{prop:covar_stoch_conv}, one obtains
$$
\E\left[ u_1^{\ell,m}(t)\right]=\cos(t(\ell(\ell+1))^{1/2})\E\left[v_1^{\ell,m}\right]
+(\ell(\ell+1))^{-1/2}\sin(t(\ell(\ell+1))^{1/2})\E\left[v_2^{\ell,m}\right]
$$
and similarly for the second component. 
This corresponds to the errors in the initial values, see the proof of Proposition~\ref{prop:error}, and we thus obtain the error bound
$$
\|\E\left[u_1(t) - u_1^\gk(t)\right]\|_{L^2(\IS^2)} 
	\le \hat{C} \cdot \left( \gk^{-\beta}\|v_1\|_{H^\beta(\IS^2)}
		+ \gk^{-(\gamma+1)}\|v_2\|_{H^\gamma(\IS^2)} \right)
$$
and correspondingly for the second component. 

In order to bound the second moments, we observe that
\begin{align*}
&\E\left[ \|u_j(t)\|^2_{L^2(\IS^2)} - \|u_j^\gk(t)\|^2_{L^2(\IS^2)}\right] =
\E\left[\langle u_j(t)+u_j^\gk(t),u_j(t)-u_j^\gk(t)\rangle_{L^2(\IS^2)} \right] \\ 
&=
\E\left[\left\langle 2\sum_{\ell=0}^\gk\sum_{m=-\ell}^\ell u_j^{\ell,m}(t)Y_{\ell,m}+
\sum_{\ell=\gk+1}^\infty\sum_{m=-\ell}^\ell u_j^{\ell,m}(t)Y_{\ell,m},
\sum_{\ell=\gk+1}^\infty\sum_{m=-\ell}^\ell u_j^{\ell,m}(t)Y_{\ell,m} \right\rangle_{L^2(\IS^2)}\right]\\
&=2\E\left[\sum_{\ell=0}^\gk\sum_{m=-\ell}^\ell\sum_{\ell'=\gk+1}^\infty\sum_{m'=-\ell'}^{\ell'} 
u_j^{\ell,m}(t)u_j^{\ell',m'}(t) \langle Y_{\ell,m}, Y_{\ell',m'}\rangle_{L^2(\IS^2)}\right]
+\E\left[\|u_j(t)-u_j^\gk(t)\|^2_{L^2(\IS^2)}\right],
\end{align*}
for $j=1,2$. Using the orthogonality of the spherical harmonics~$\cY$, the first term vanishes and the second one is bounded by the square of the strong error in Proposition~\ref{prop:error}. This yields
$$
\E\left[ \|u_1(t)\|^2_{L^2(\IS^2)} - \|u_1^\gk(t)\|^2_{L^2(\IS^2)}\right] \le 
\hat{C} \cdot \left( \gk^{-\ga}+\gk^{-2\beta}\|v_1\|_{H^\beta(\IS^2)}+
\gk^{-2(\gamma+1)}\|v_2\|_{H^\gamma(\IS^2)} \right)
$$
and similarly for the second component. 
\end{proof}

For a more general class of test functions, we obtain weak error rates that depend directly on the regularity of the test function and indirectly on the regularity of the solution. Let us first state the abstract assumption on the test functions that will be required for the next weak convergence result.

\begin{assumption}\label{ass:test_functions}
	Consider the class of Fr\'echet differentiable test functions $\varphi$ satisfying for some fixed $s>0$
	\begin{equation*}
		\|\int_0^1\varphi'(\rho u_j(t)+(1-\rho)u_j^\gk(t))\,\diff\rho\|_{L^2(\Omega;H^s(\IS^2))}
			\le \widetilde{C}
			< + \infty
	\end{equation*}
	for $j=1,2$.
\end{assumption}
A typical example of a set of test functions satisfying the above assumption 
would be polynomial growth of the derivative in~$H^s(\IS^2)$, i.\,e., to take~$\varphi$ such that for all $x \in H^s(\IS^2)$
\begin{equation}\label{eq:poly_growth}
\|\varphi'(x)\|_{H^s(\IS^2)}
\le C \left(1 + \|x\|_{H^s(\IS^2)}^m\right).
\end{equation}
Then we observe that
\begin{equation*}
 \rho u_j(t)+(1-\rho)u_j^\gk(t)
 	= \sum_{\ell=0}^\gk \sum_{m=-\ell}^\ell u_j^{\ell,m}(t)Y_{\ell, m}
 		+ \rho \sum_{\ell=\gk+1}^\infty\sum_{m=-\ell}^\ell u_1^{\ell,m}(t)Y_{\ell, m}.
\end{equation*}
This would imply for $\rho \in [0,1]$
\begin{align*}
 & \|\varphi'(\rho u_j(t)+(1-\rho)u_j^\gk(t))\|_{L^2(\Omega;H^s(\IS^2))}^2\\
 	& \le C^2 \E \left[ \left( 1 + \|\rho u_j(t)+(1-\rho)u_j^\gk(t)\|_{H^s(\IS^2)}^m \right)^2 \right]\\
 	& \le 2 C^2 \left( 1 + 
 		\E \left[ \left( 
 			\left\|\sum_{\ell=0}^\gk \sum_{m=-\ell}^\ell u_j^{\ell,m}(t)Y_{\ell, m}\right\|_{H^s(\IS^2)} 
 				+ \rho \left\| \sum_{\ell=\gk+1}^\infty\sum_{m=-\ell}^\ell u_1^{\ell,m}(t)Y_{\ell, m}\right\|_{H^s(\IS^2)} 
 				\right)^{2m} \right]
 				\right)\\
 	& \le 2 C^2 \left( 1 + \|u_j(t)\|_{L^{2m}(\gO;H^s(\IS^2))}^{2m} \right).
\end{align*}
Therefore Assumption~\ref{ass:test_functions} is satisfied 
if $u_j(t) \in L^{2m}(\gO;H^s(\IS^2))$ which is specified in Proposition~\ref{prop:regularity}. 

Having seen that the class of test functions with derivatives of polynomial growth satisfies Assumption~\ref{ass:test_functions}, we are in place to state our general weak convergence result.

\begin{proposition}\label{prop:weak2}
Under the setting of Proposition~\ref{prop:weak} and Assumption~\ref{ass:test_functions}, there exists a constant~$\hat{C}$ such that the weak errors are bounded by
\begin{align*}
\left|\E\left[\varphi(u_1(t)) - \varphi(u_1^\gk(t))\right]\right| 
	&\le \hat{C}\left( \gk^{-(\alpha/2+s)}
		+\gk^{-(\beta+s)}\|v_1\|_{H^\beta(\IS^2)}
		+\gk^{-(\gamma+s+1)}\|v_2\|_{H^\gamma(\IS^2)} \right) \\
\left|\E\left[\varphi(u_2(t)) - \varphi(u_2^\gk(t))\right]\right| 
	&\le \hat{C}\left( \gk^{-(\alpha/2+s-1)}
		+\gk^{-(\beta+s-1)}\|v_1\|_{H^\beta(\IS^2)}
		+\gk^{-(\gamma+s)}\|v_2\|_{H^\gamma(\IS^2)} \right) 
\end{align*}
for all $\gk > \ell_0$.
\end{proposition}

\begin{proof}
The proof is inspired by~\cite{AKL16}.
Consider the Gelfand triple 
$$
V\subset H\subset V^*
$$
with $V=H^s(\IS^2), H=L^2(\IS^2)$ and $V^*=H^{-s}(\IS^2)$. The mean value theorem for Fr\'echet derivatives followed by the Cauchy--Schwarz inequality yields
\begin{align*}
\left|\E\left[ \varphi(u_j(t))-\varphi(u_j^\gk(t)) \right]\right|
	&= \left| \E\left[ \prescript{}{V}\langle \int_0^1\varphi'(\rho u_j(t)+(1-\rho)u_j^\gk(t))\,\text d\rho, u_j(t)-u_j^\gk(t)\rangle_{V^*} \right] \right| \\
&\leq \left\|\int_0^1\varphi'(\rho u_j(t)+(1-\rho)u_j^\gk(t))\,\text d\rho\right\|_{L^2(\Omega;V)} \|u_j(t)-u_j^\gk(t)\|_{L^2(\Omega;V^*)}
\end{align*}
for $j=1,2$. The first term is bounded by Assumption~\ref{ass:test_functions} 
so that the convergence rate will be obtained from the second term. 
Details are only given for the first component, i.\,e., for $j=1$, 
and are obtained for $u_2$ in a similar way. 

Following the proof of Proposition~\ref{prop:error}, we obtain
\begin{align*}
&\left\|u_1(t)-u_1^\gk(t)\right\|_{L^2(\gO;H^{-s}(\IS^2))}\\
&\qquad \leq \|v_1 - v_1^\gk\|_{H^{-s}(\IS^2)} 
		+ \left(\sum_{\ell=\gk+1}^\infty \sum_{m=-\ell}^\ell (1+\ell(\ell+1))^{-s} (\ell(\ell+1))^{-1}|v_2^{\ell,m}|^2\right)^{1/2}\\
		& \qquad \qquad + \|\hat{Z}-\hat{Z}^\gk\|_{L^2(\gO;L^2(\IS^2))}
\end{align*}
with
\begin{equation*}
 \hat{Z} = (\text{Id}-\Delta_{\IS^2})^{-s/2} Z
\end{equation*}
and $\hat{Z}^\gk$ its approximation. Therefore the angular power spectrum of the centered Gaussian random field~$\hat{Z}$ is given by
\begin{equation*}
 \hat{A}_\ell
 	= (1 + \ell(\ell+1))^{-s} \widetilde{A}_\ell
 	\le C \ell^{-(\alpha + 2 + 2s)}.
\end{equation*}
Applying Theorem~\ref{thm:iGRF_Lp_conv} to~$\hat{Z}$ and bounding the initial conditions as in Proposition~\ref{prop:error} with the additional weights $(1+ \ell(\ell+1))^{-s/2}$ yield
\begin{align*}
 \left\|u_1(t)-u_1^\gk(t)\right\|_{L^2(\gO;H^{-s}(\IS^2))}
	&\leq C\left( \gk^{-(\alpha/2+s)}
		+\gk^{-(\beta+s)}\|v_1\|_{H^\beta(\IS^2)}
		+\gk^{-(\gamma+s+1)}\|v_2\|_{H^\gamma(\IS^2)} \right), 	
\end{align*}
which concludes the proof for the weak error in the first component.
\end{proof}

Proposition~\ref{prop:weak} states that the approximation of the second moment converges with twice the strong rate of convergence obtained in Proposition~\ref{prop:error}. Let us now investigate which regularity (for the noise and initial values) is required to achieve twice the strong rate in Proposition~\ref{prop:weak2}. We focus on the convergence of the noise in the parameter~$\ga$. Similar considerations hold for the initial conditions.

For having the weak rate in Proposition~\ref{prop:weak2} to be twice the strong rate from Proposition~\ref{prop:error}, 
one would need $s = \ga/2$ for $u_1$ and $s = \ga/2-1$ for $u_2$.

The regularity result from Proposition~\ref{prop:regularity} reads $u_1(t) \in L^{2m}(\gO;H^s(\IS^2))$ for all $s < \ga/2$ and $u_2(t) \in L^{2m}(\gO;H^s(\IS^2))$ for all $s < \ga/2-1$. This together with the polynomial growth 
assumption~\eqref{eq:poly_growth} on the test functions would imply that Assumption~\ref{ass:test_functions} 
is satisfied for all $s < \ga/2$ for $u_1$ and $s < \ga/2-1$ for $u_2$. Therefore, in the situation of Proposition~\ref{prop:weak2},
the general rule of thumb for the rate of weak convergence is also valid.

We end this section by observing that several strategies for proving weak rates of convergence 
of numerical solutions to SPDEs in the literature could be extended to the present setting or 
in the case of numerical discretizations of nonlinear stochastic wave equations on the sphere, 
see for instance \cite{DP09,KLL12,W15,BHS18,JJW18,HM19,KLP20} and references therein. This could be subject of future research.

\section{Numerical experiments}\label{sec:num}
We present several numerical experiments with the aim of supporting and illustrating  
the above theoretical results. 

In order to illustrate the rate of convergence of the mean-square error from Proposition~\ref{prop:error}, 
we consider a ``reference'' solution at time $T=1$ with $\gk=2^7$ (since for larger $\gk$ the elements 
of the angular power spectrum~$A_\ell$ and therefore the increments 
were so small that MATLAB failed to calculate the series expansion). 
The initial values are taken to be $v_1=v_2=0$ in order to observe the convergence rate only with respect to the regularity of the noise given by the parameter~$\ga$. Afterwards we will perform a numerical example illustrating the convergence rate with respect to the regularity of the initial position given by the parameter~$\gb$.  
We then compute one time step of numerical solutions 
(since we have shown in Proposition~\ref{prop:error} that the convergence rate is independent 
of the number of calculated time steps) and compute the errors for various truncation indices~$\gk$. 
Instead of the $L^2(\IS^2)$ error in space, we used the maximum over all grid points which is a stronger error. 
The results and the theoretical convergence rates are shown for $\ga=3$ and $\ga=5$ in Figure~\ref{fig:swe1a}, 
resp.\ Figure~\ref{fig:swe1b}. 

\begin{figure}
\subfigure[Sample of solution.]{\includegraphics[width=0.32\textwidth]{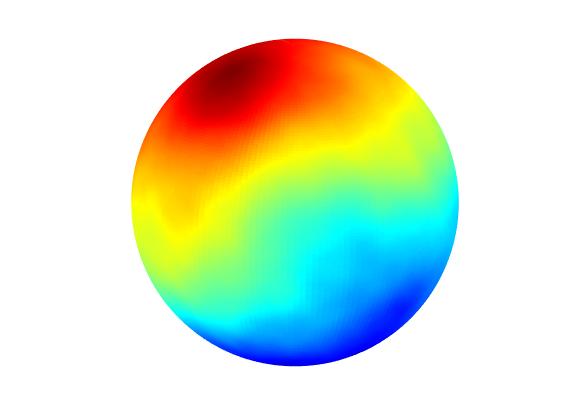}}
\subfigure[Errors in position.]{\includegraphics[width=0.32\textwidth]{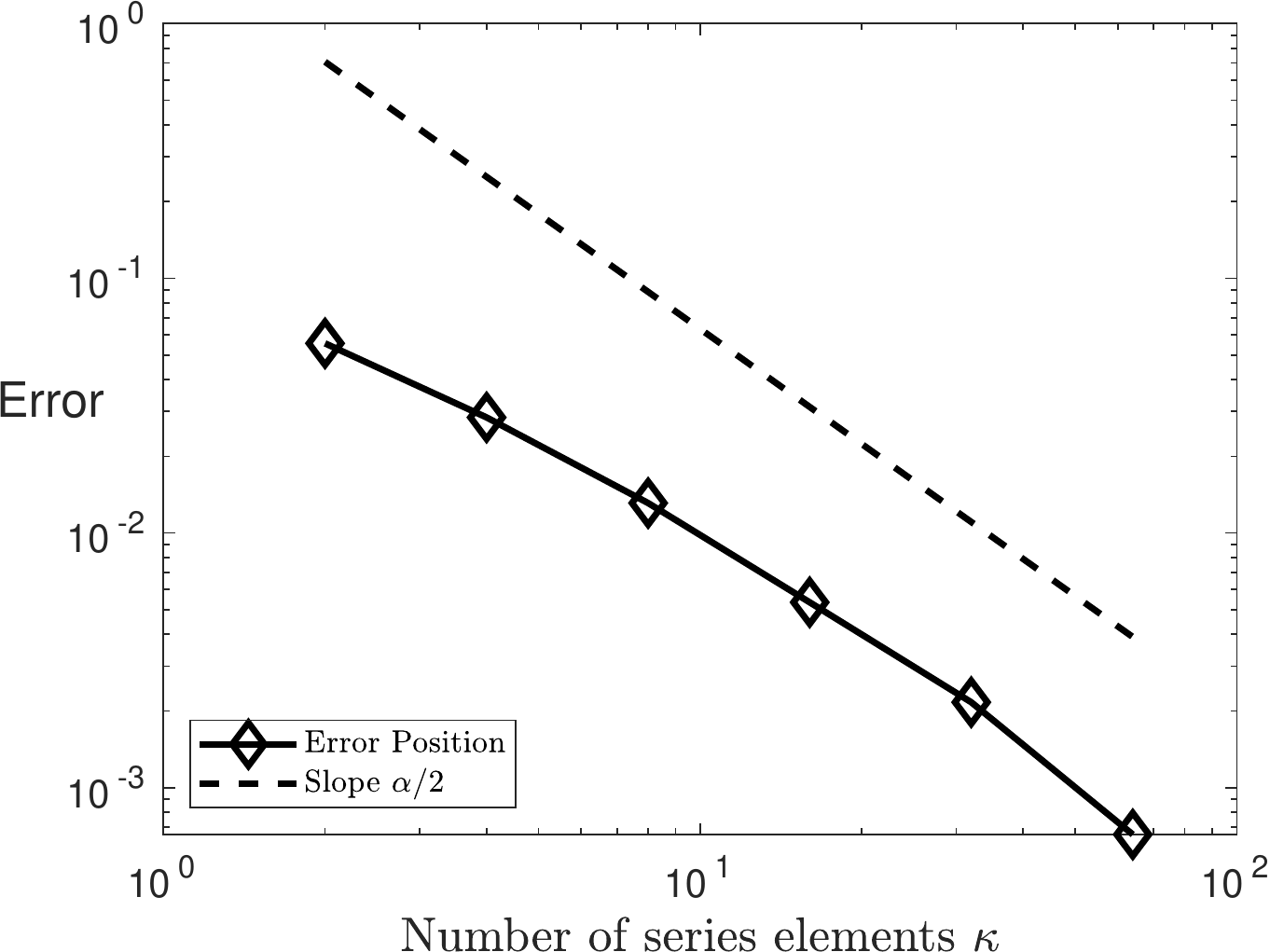}}
\subfigure[Errors in velocity.]{\includegraphics[width=0.32\textwidth]{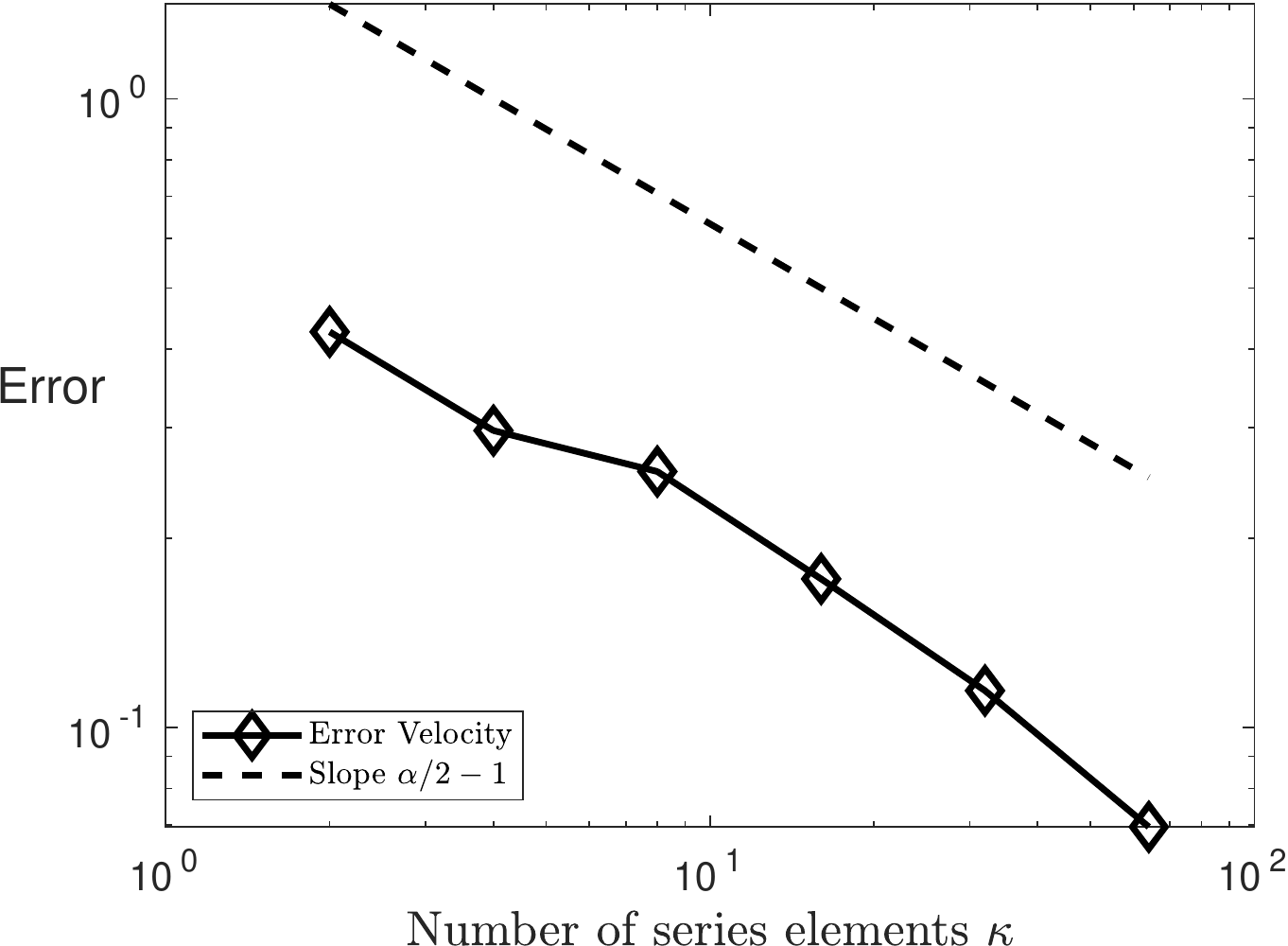}}
\caption{Sample and mean-square errors of the approximation of the stochastic wave equation 
with angular power spectrum of the $Q$-Wiener process with parameter $\ga=3$ and $100$~Monte Carlo samples.}
\label{fig:swe1a}
\end{figure}

In these figures, one observes that the simulation results match 
the theoretical results from Proposition~\ref{prop:error}. 
In addition, in order to illustrate the structure of the solution~$u$   
in dependence of the decay of the angular power spectrum, we include samples 
next to the convergence plots.

\begin{figure}
\subfigure[Sample of solution.]{\includegraphics[width=0.32\textwidth]{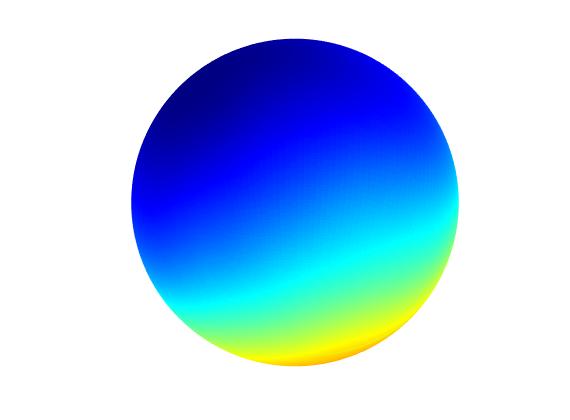}}
\subfigure[Errors in position.]{\includegraphics[width=0.32\textwidth]{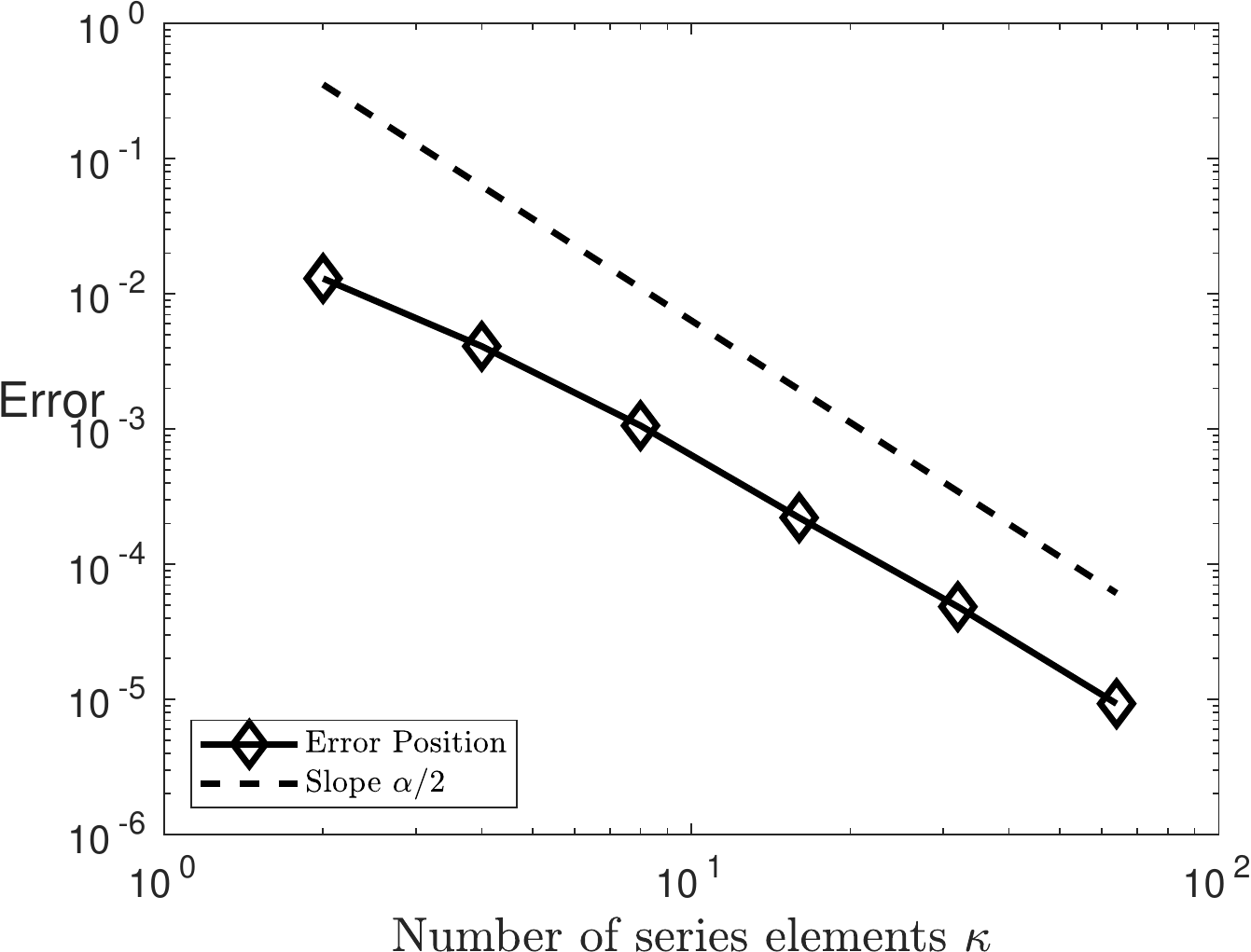}}
\subfigure[Errors in velocity.]{\includegraphics[width=0.32\textwidth]{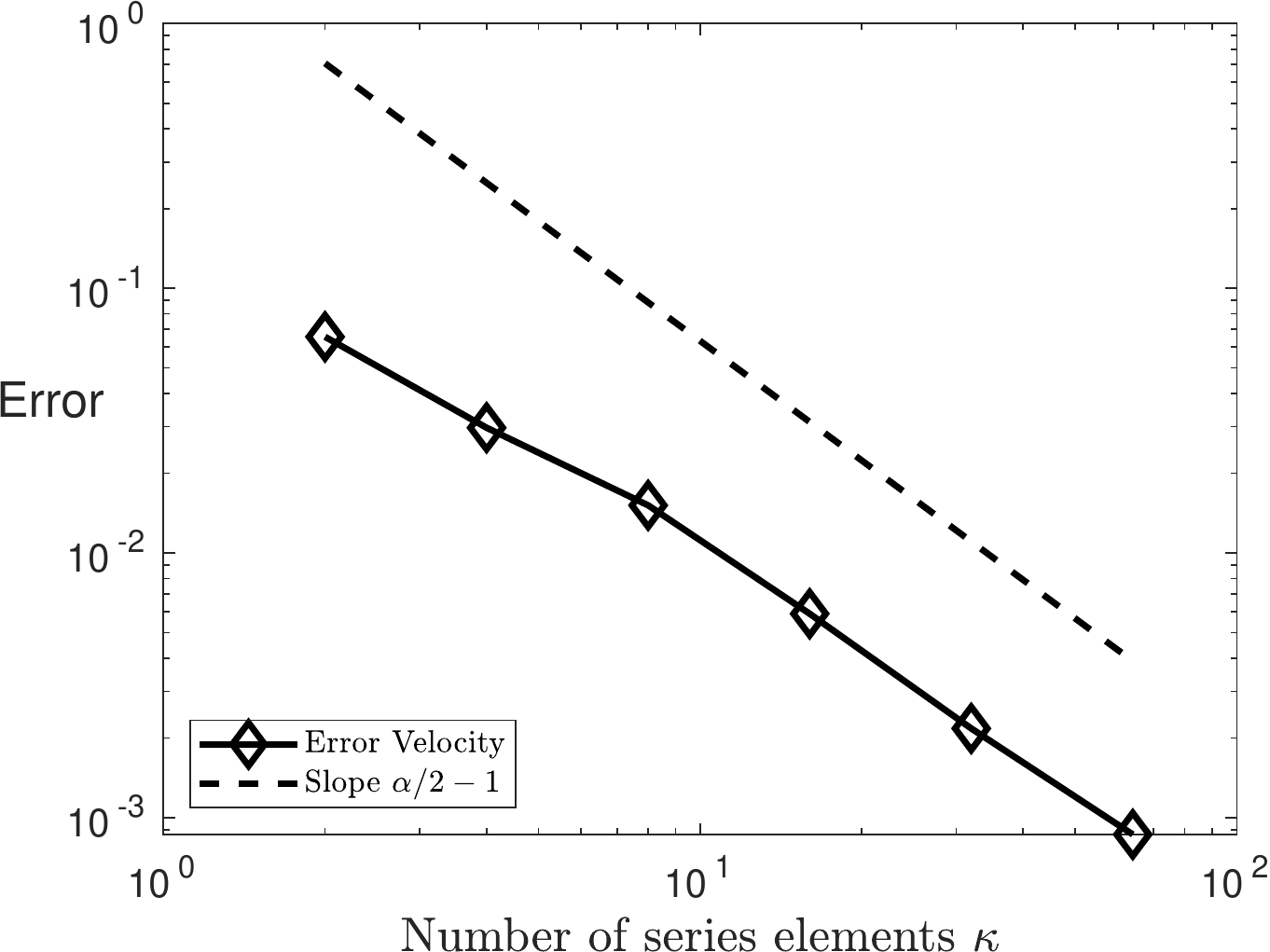}}
\caption{Sample and mean-square errors of the approximation of the stochastic wave equation 
with angular power spectrum of the $Q$-Wiener process with parameter $\ga=5$ and $100$~Monte Carlo samples.}
\label{fig:swe1b}
\end{figure}

In order to illustrate Remark~\ref{rem1} on the possibility of taking the parameter $0<\alpha<2$ 
to show convergence in the first component, we repeat the previous numerical experiments with $\ga=1$. 
The results are presented in Figure~\ref{fig:swe1c}. There, for such non-smooth noise, 
one can observe convergence in the position but not in the velocity. 
Similar observations were made for time discretizations of stochastic wave equations on 
domains (that are not manifolds) in \cite{MR3033008,MR3484400}, for instance.

\begin{figure}
\subfigure[Sample of solution.]{\includegraphics[width=0.32\textwidth]{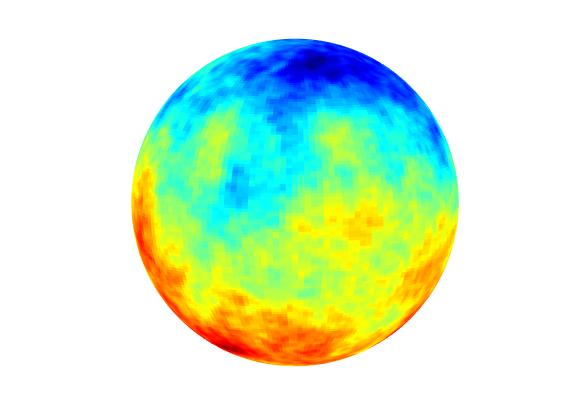}}
\subfigure[Errors in position.]{\includegraphics[width=0.32\textwidth]{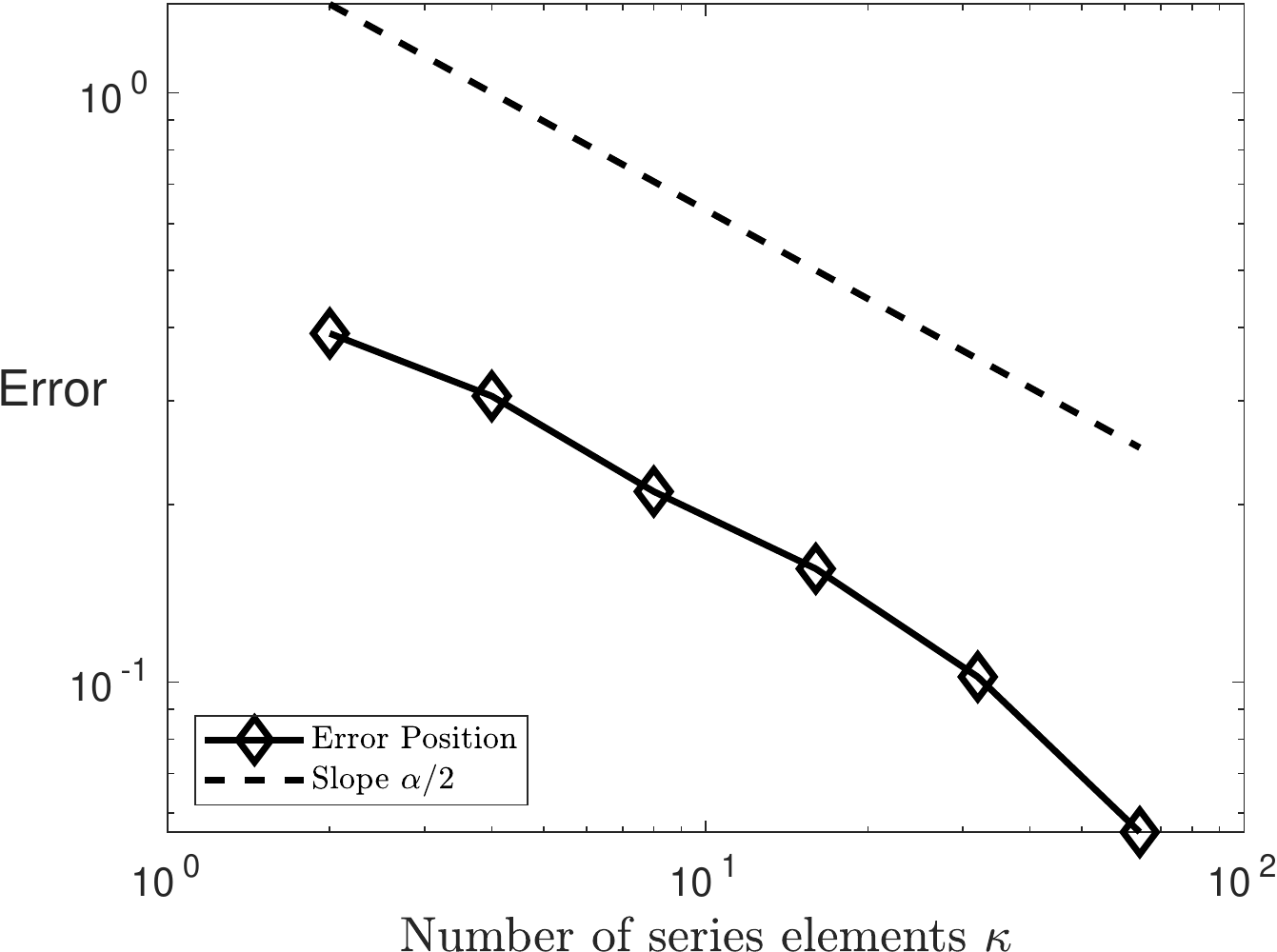}}
\subfigure[Errors in velocity.]{\includegraphics[width=0.32\textwidth]{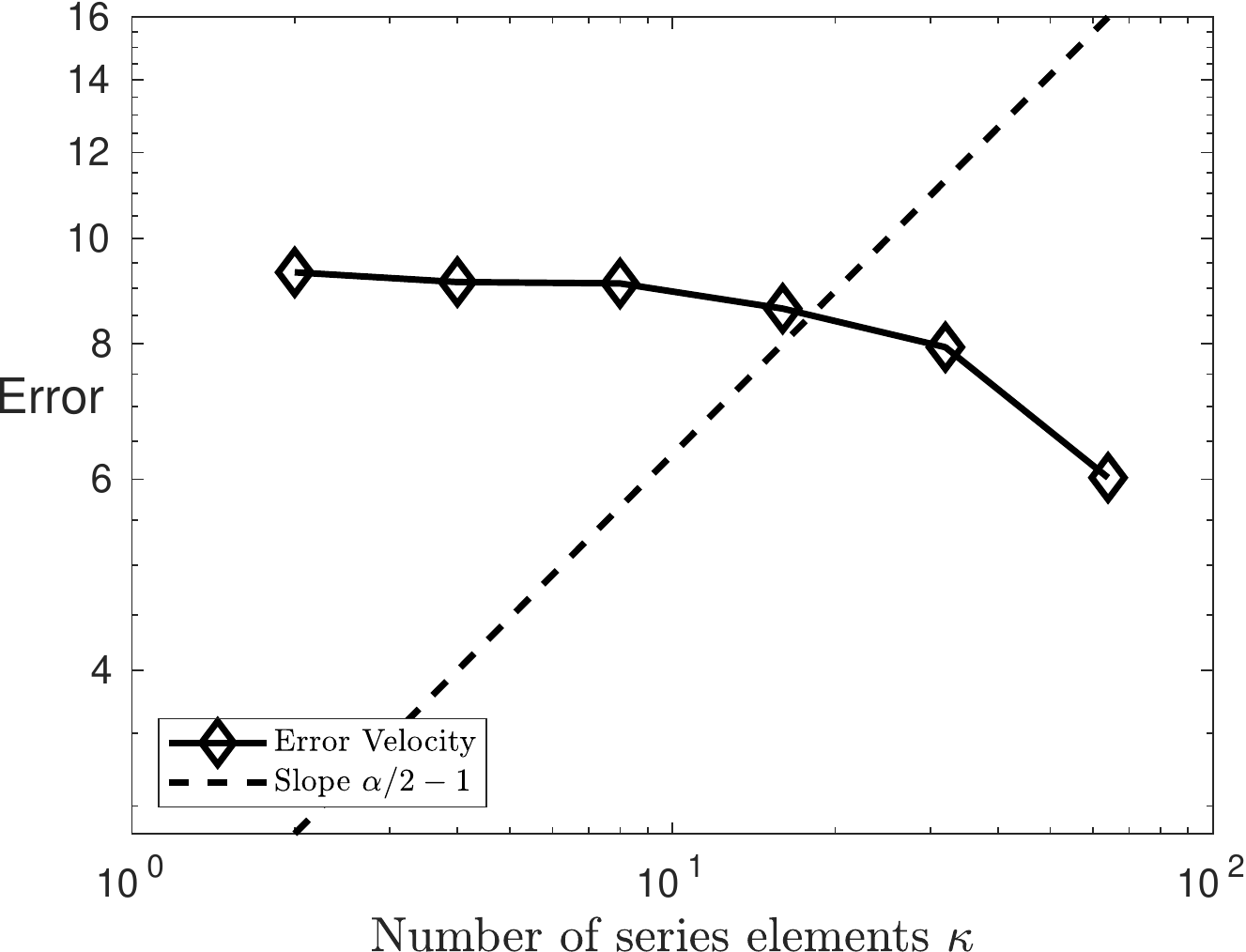}}
\caption{Sample and mean-square errors of the approximation of the stochastic wave equation 
with angular power spectrum of the $Q$-Wiener process with parameter $\ga=1$ and $100$~Monte Carlo samples.}
\label{fig:swe1c}
\end{figure}

In Figure~\ref{fig:init} we illustrate the convergence rates with respect to the regularity of the initial position from Proposition~\ref{prop:error}. To ensure that the regularity of the initial position dominates the error, we choose $\ga=10$ and a random initial position~$v_1$ scaled such that it belongs to $H^\gb(\IS^2)$ with $\gb=2$. The expected convergence rates are indeed observed in this figure.

\begin{figure}
	\subfigure[Errors in position.]{\includegraphics[width=0.49\textwidth]{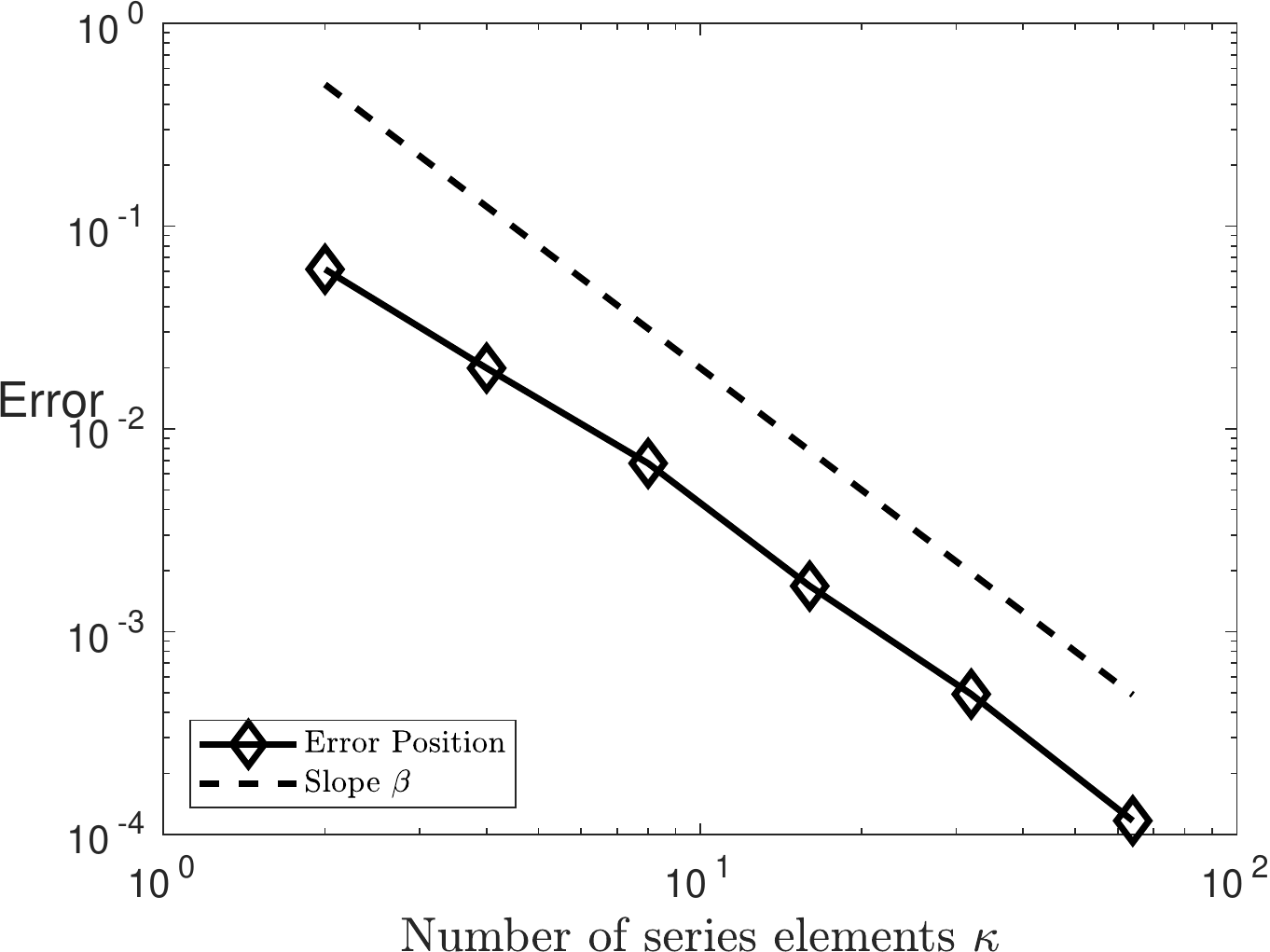}}
	\subfigure[Errors in velocity.]{\includegraphics[width=0.49\textwidth]{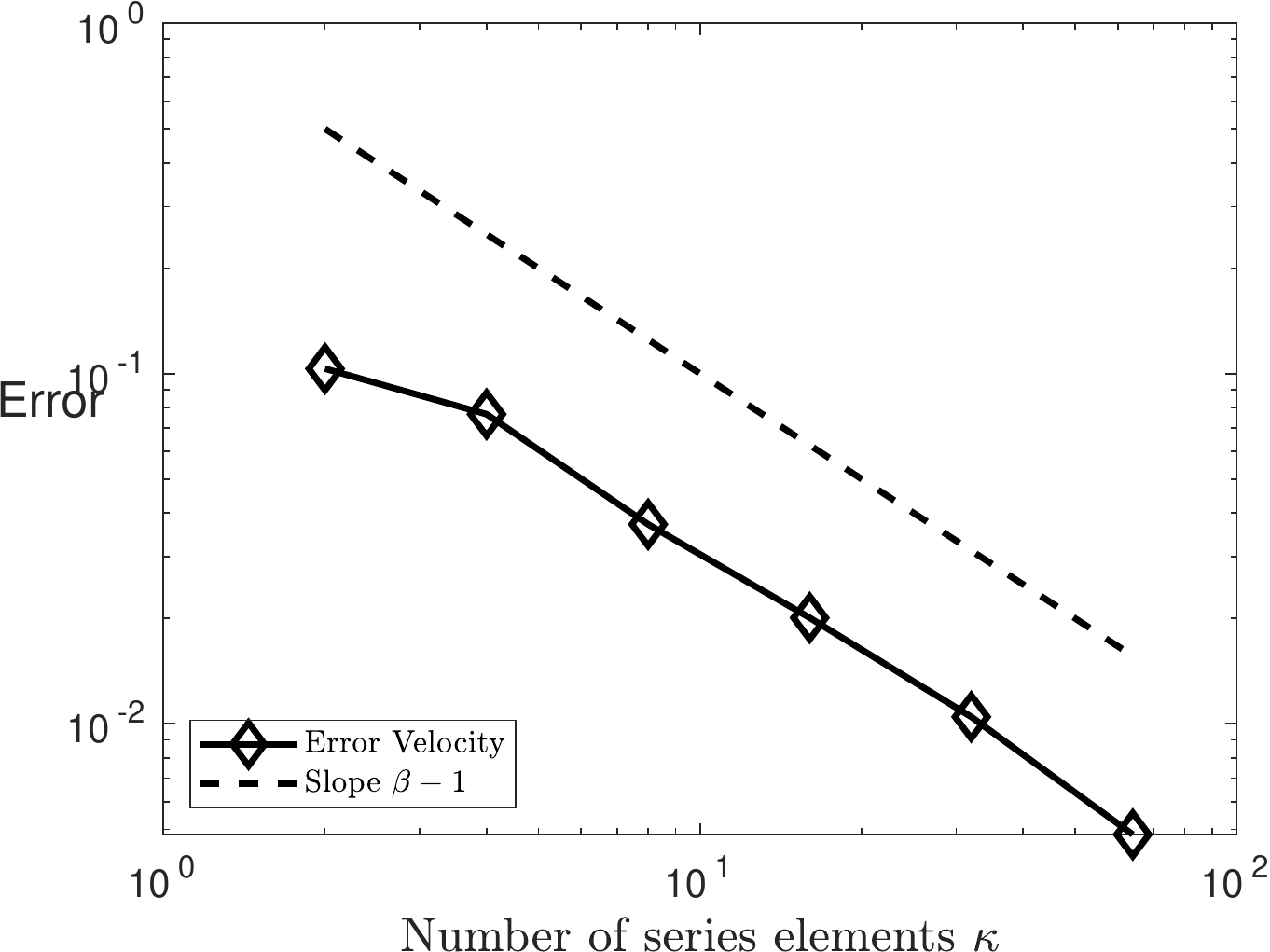}}
	\caption{Mean-square errors of the approximation of the stochastic wave equation 
		with angular power spectrum of the $Q$-Wiener process with parameter $\ga=10$ and $v_1 \in H^\gb(\IS^2)$ for $\gb = 2$ and $100$~Monte Carlo samples.}\label{fig:init}
\end{figure}

Errors of one path of the stochastic wave equation to the corresponding error plots from the previous figures 
(Figure~\ref{fig:swe1a} and Figure~\ref{fig:swe1b})  
are presented in Figure~\ref{fig:swe2}. The observed convergence rates coincide with 
the theoretical results on $\IP$-almost sure convergence in the second part of Proposition~\ref{prop:error}.

\begin{figure}
\subfigure[Angular power spectrum with parameter $\ga=3$.]{\includegraphics[width=0.49\textwidth]{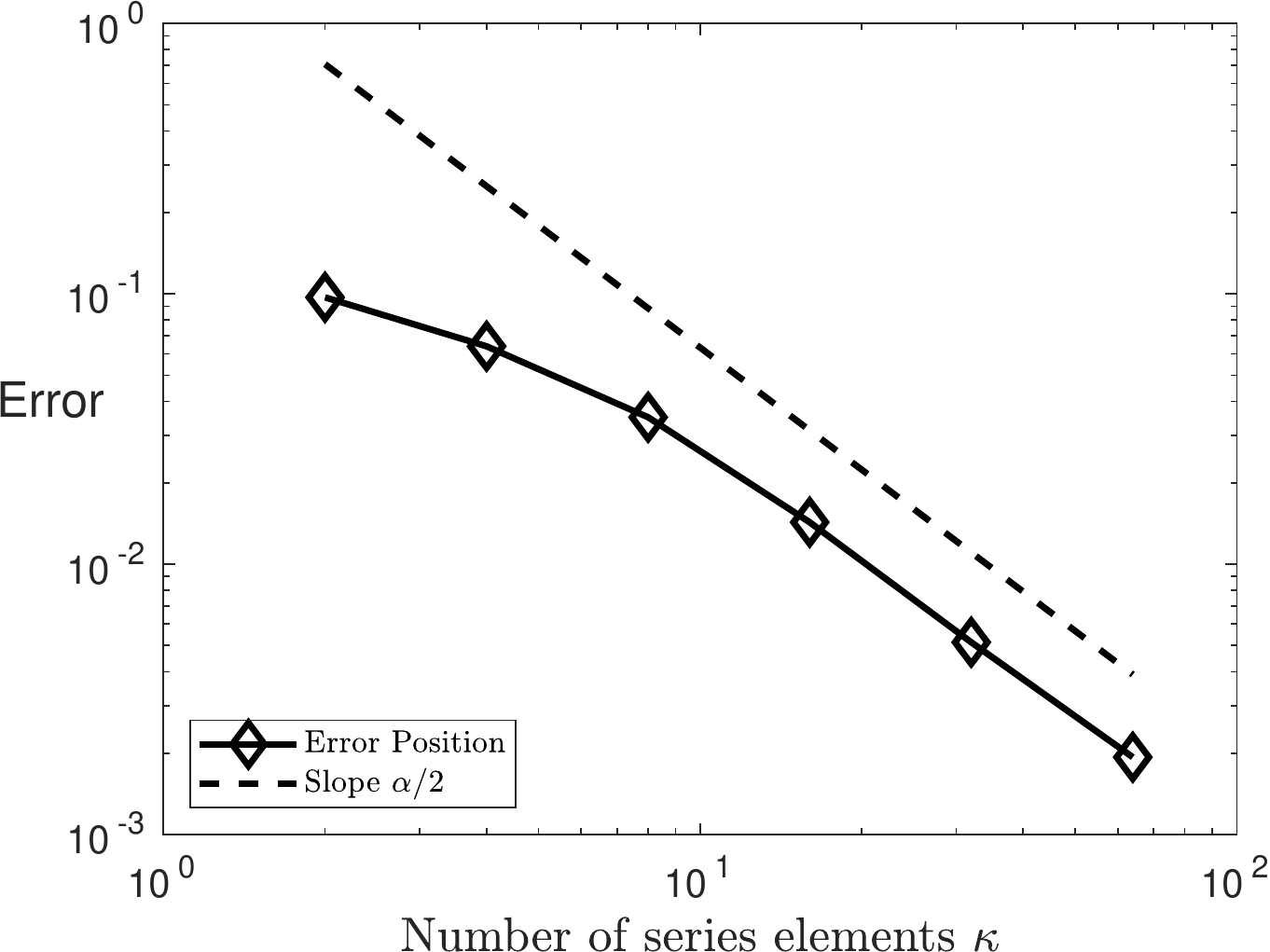}}
\subfigure[Angular power spectrum with parameter $\ga=5$.]{\includegraphics[width=0.49\textwidth]{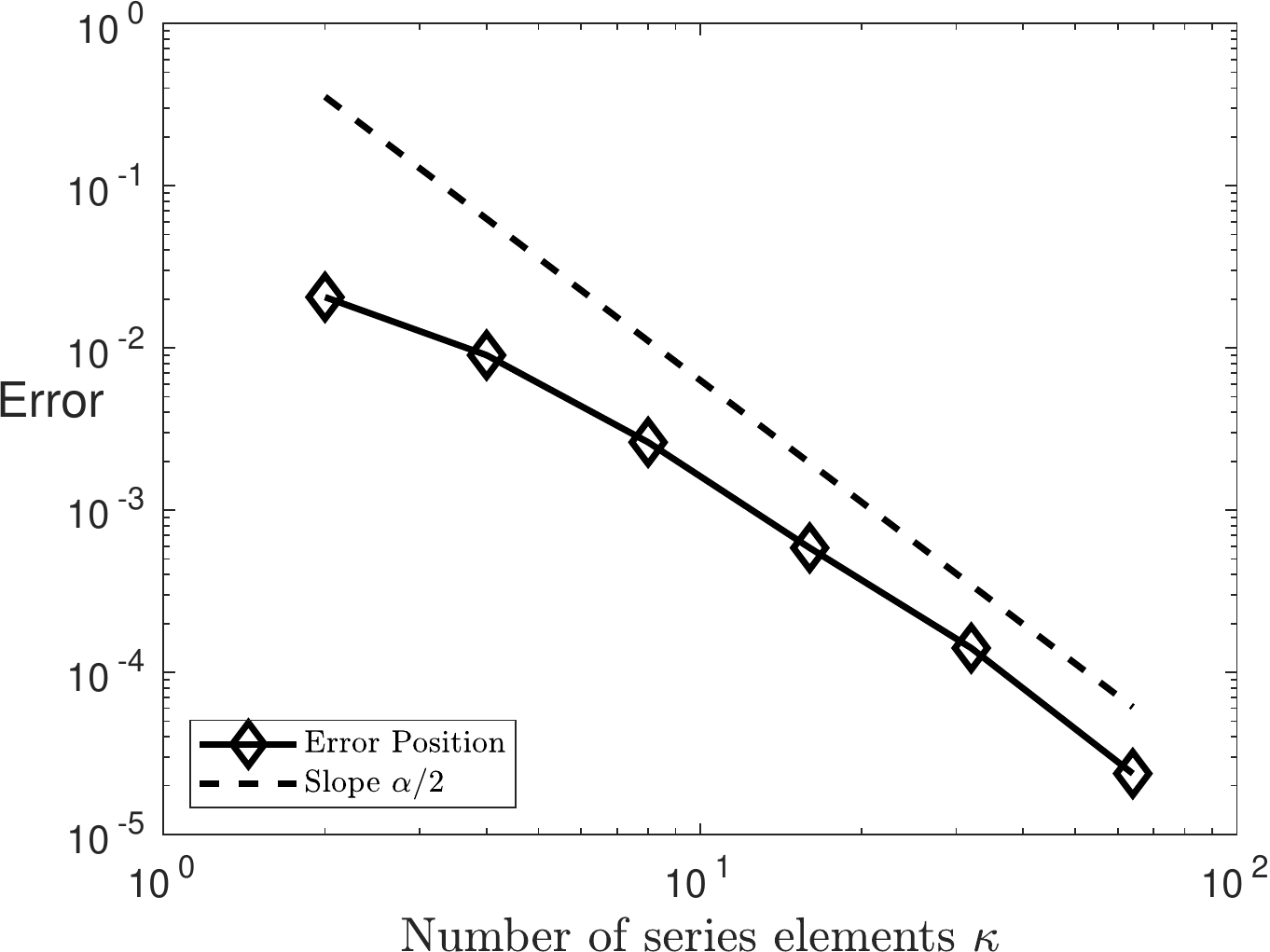}}
  \caption{Error of the approximation of a path of the stochastic wave equation 
  with different angular power spectra of the $Q$-Wiener process.}\label{fig:swe2}
\end{figure}

Let us now illustrate the weak rates of convergence from Proposition~\ref{prop:weak} and Proposition~\ref{prop:weak2}.
We consider a ``reference'' solution at time $T=1$ with $\gk=2^7$. 
The initial values are taken to be $v_1=v_2=0$. The test functions are given by $\varphi(u)=\|u\|_{L^2(\IS^2)}^2$ and 
$\varphi(u)=\exp(-\|u\|_{L^2(\IS^2)}^2)$. Observe that the second test function is of class~$C^2$, bounded and with bounded derivatives. Proposition~\ref{prop:weak} and Proposition~\ref{prop:weak2} guarantee that the weak rates will be essentially twice the strong rates in both cases. This is confirmed for $\ga=3$ in Figure~\ref{fig:weak}.

\begin{figure}
\subfigure[$\varphi(u)=\|u\|_{L^2(\IS^2)}^2$.]{\includegraphics[width=0.49\textwidth]{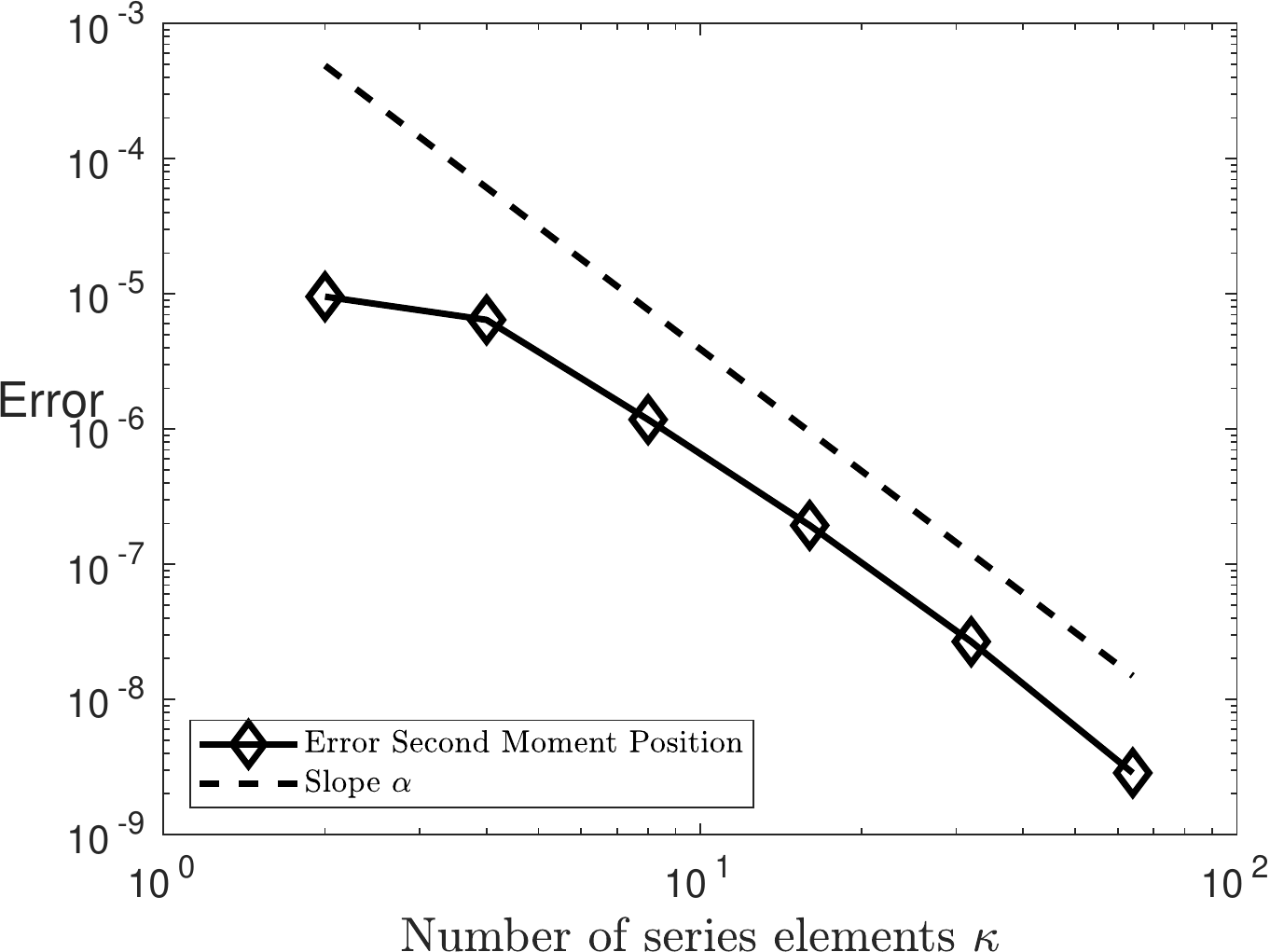}}
\subfigure[$\varphi(u)=\|u\|_{L^2(\IS^2)}^2$.]{\includegraphics[width=0.49\textwidth]{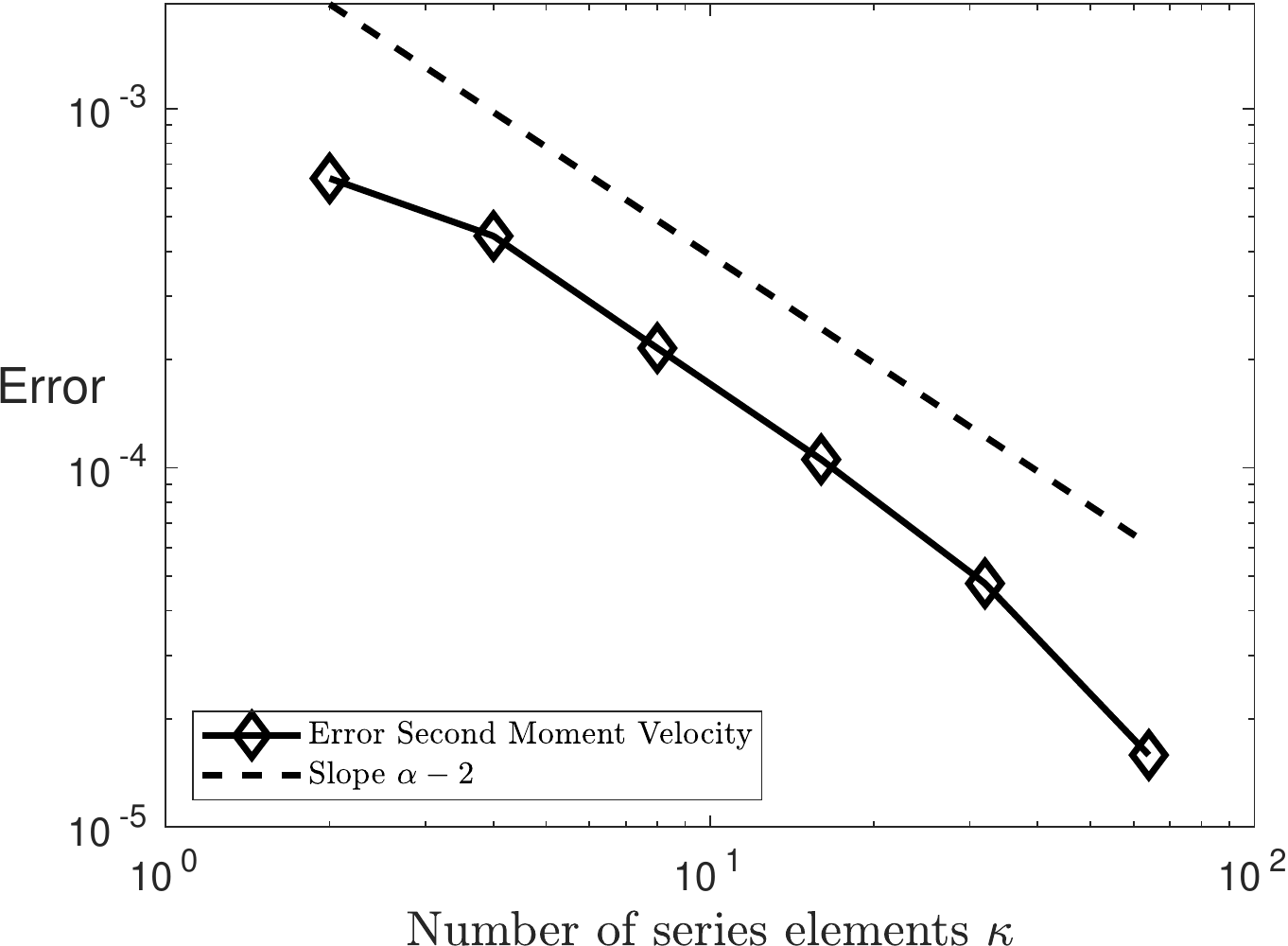}}
\subfigure[$\varphi(u)=\exp(-\|u\|_{L^2(\IS^2)}^2)$.]{\includegraphics[width=0.49\textwidth]{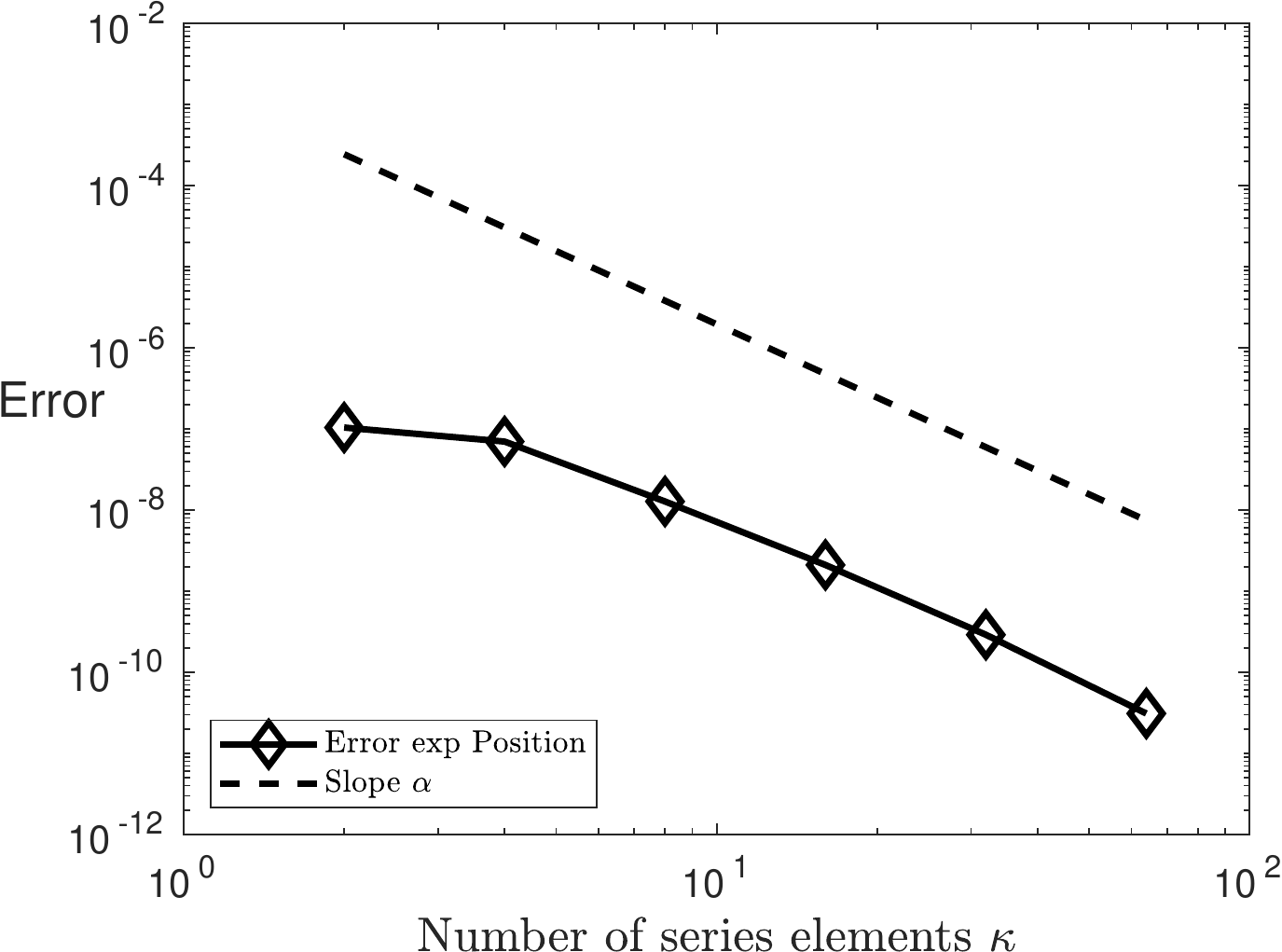}}
\subfigure[$\varphi(u)=\exp(-\|u\|_{L^2(\IS^2)}^2)$.]{\includegraphics[width=0.49\textwidth]{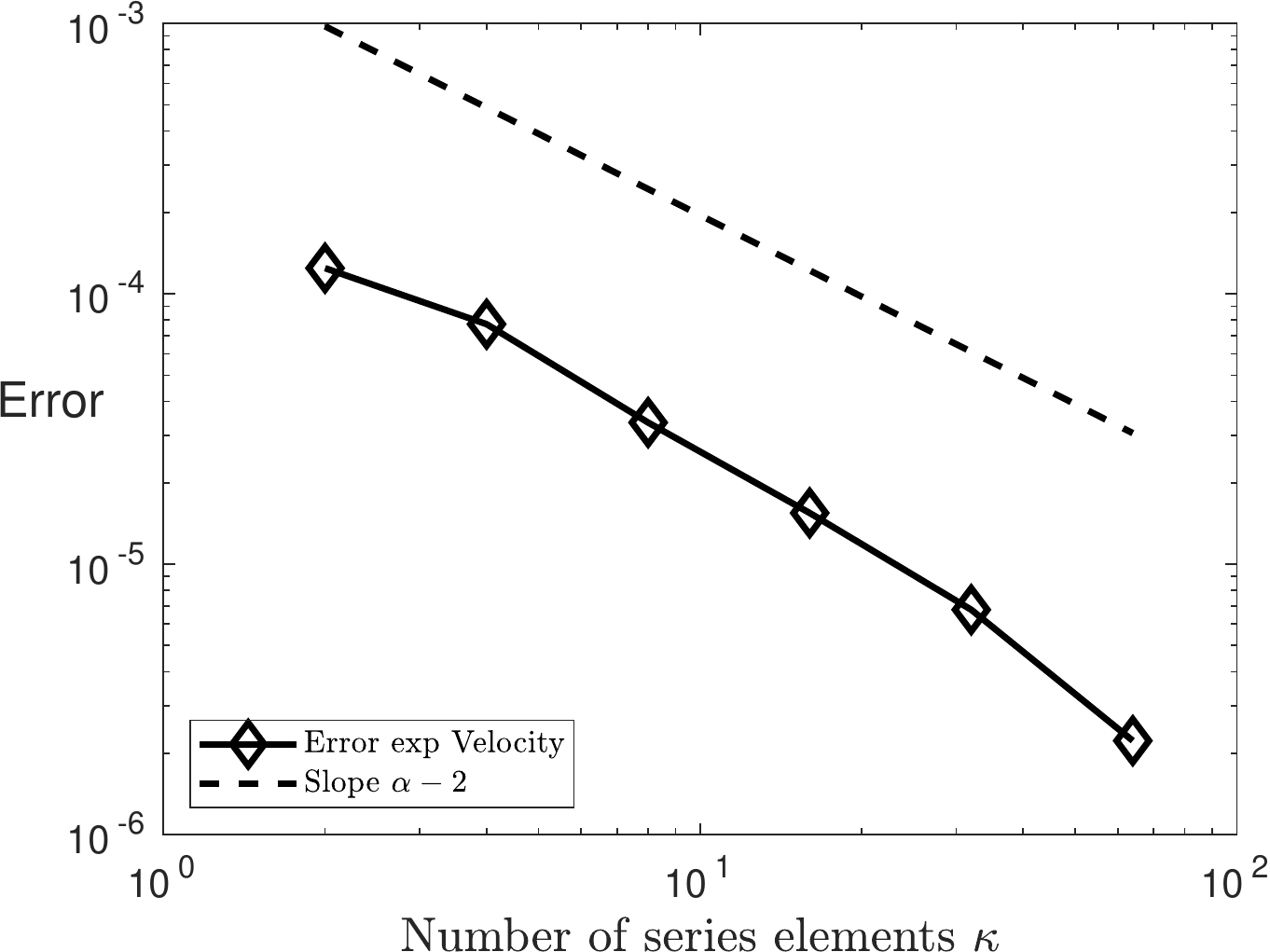}}
\caption{Weak errors of the approximation of the stochastic wave equation 
with angular power spectrum of the $Q$-Wiener process with parameter $\ga=3$ 
and $1000$~Monte Carlo samples. Left column shows position, right column velocity.}
\label{fig:weak}
\end{figure}

\section{Further extensions}%\label{sec:ext}
In this section, we extend some of the above results first to the case of the stochastic wave equation on  
higher-dimensional spheres $\IS^{d-1}$, for some integer $d>3$, and second to the case of a free 
stochastic Schr\"odinger equation on the sphere $\IS^2$. We keep this section concise and focus on strong and $\IP$-a.s.\ convergence.

\subsection{The stochastic wave equation on~$\IS^{d-1}$}

Let us consider the more general situation of the stochastic wave equation on the unit sphere $\IS^{d-1} = \{x \in \R^d, \|x\|_{\R^d} = 1\}$ embedded into $\R^d$. 
The angular distance of two points $x$ and $y$ on~$\IS^{d-1}$ is given in the same way as on~$\IS^2$, see Section~\ref{sec:setting}. 
Let us denote by $(S_{\ell, m}, \ell \in \N_0, m=1,\ldots,h(\ell,d))$ the spherical harmonics on~$\IS^{d-1}$, where
\begin{equation*}
h(\ell,d) = (2\ell + d - 2) \frac{(\ell+d-3)!}{(d-2)!\,\ell!}.
\end{equation*}

Using the same setup as in~\cite{MR3404631} which goes back to~\cite{Y83}, a centered isotropic Gaussian random field~$Z$ on~$\IS^{d-1}$ admits a Karhunen--Lo\`eve expansion
\begin{equation*}
Z(x) = \sum_{\ell=0}^\infty \sum_{m=1}^{h(\ell,d)} a_{\ell, m} S_{\ell, m}(x),
\end{equation*}
where $(a_{\ell, m}, \ell \in \N_0,m=1,\ldots,h(\ell,d))$ is a sequence of independent Gaussian random variables satisfying
\begin{equation*}
\E[a_{\ell, m}] = 0, \quad
\E[a_{\ell, m} a_{\ell', m'}] = A_\ell \delta_{\ell \ell'} \delta_{m m'}
\end{equation*}
for $\ell, \ell' \in \N_0$ and $m=1,\ldots,h(\ell,d)$, $m' = 1,\ldots, h(\ell',d)$ and
\begin{equation*}
\sum_{\ell=0}^\infty A_\ell \, h(\ell,d) < + \infty.
\end{equation*}
The series converges with probability one and in $L^p(\gO;\R)$ 
as well as in $L^2(\gO;L^p(\IS^{d-1}))$, $p \ge 1$. 
Denoting by $(A_\ell, \ell \in \N_0)$ the angular power spectrum of~$Z$
for $\IS^{d-1}$ in analogy to what was done for~$\IS^2$, 
we can rewrite
\begin{equation*}
Z
= \sum_{\ell=0}^\infty \sum_{m=1}^{h(\ell,d)} a_{\ell, m} S_{\ell, m}
= \sum_{\ell=0}^\infty \sqrt{A_\ell} \sum_{m=1}^{h(\ell,d)} X_{\ell, m} S_{\ell, m},
\end{equation*}
where $(X_{\ell, m}, \ell \in \N_0, m=1,\ldots,h(\ell,d))$ is the sequence of independent, standard normally distributed random variables derived by $X_{\ell, m} = a_{\ell, m}/\sqrt{A_\ell}$.
We set
\begin{equation*}
Z^\gk
= \sum_{\ell=0}^\gk \sqrt{A_\ell} \sum_{m=1}^{h(\ell,d)} X_{\ell, m} S_{\ell, m}
\end{equation*}
for the corresponding sequence of truncated random fields $(Z^\gk,\gk \in \N)$.
It is shown in Theorem~5.5 in~\cite{MR3404631} that these approximations converge to the random field~$Z$ in $L^p(\gO;L^2(\IS^{d-1}))$ and $\IP$-almost surely with error bounds
	\begin{equation}\label{eq:conv_Lp_RF_Sd}
	\|Z - Z^\gk\|_{L^p(\gO;L^2(\IS^{d-1}))}
	\le C_p \cdot \gk^{-(\ga+1-d)/2}
	\end{equation}
	for $\gk > \ell_0$ and
	for all $\gd < (\ga+1-d)/2$ 
	\begin{equation}\label{eq:conv_as_RF_Sd}
	\|Z - Z^\gk\|_{L^2(\IS^{d-1})}
	\le \gk^{-\gd},
	\quad \IP\text{-a.s.},
	\end{equation}
	where $A_\ell \le C \cdot \ell^{-\ga}$ for $\ell \ge \ell_0$.
	This generalizes Theorem~\ref{thm:iGRF_Lp_conv} above and leads to convergence rates 
	that depend also on the dimension of the sphere.
	
Similarly to~\eqref{eqW} in Section~\ref{sec:setting}, we introduce a $Q$-Wiener process $(W(t), t \in \IT)$ on some finite interval $\IT = [0,T]$ with values in~$L^2(\IS^{d-1})$ by the expansion 
\begin{equation}\label{eqW_Sd}
W(t,y)
= \sum_{\ell=0}^\infty \sum_{m=1}^{h(\ell,d)} a^{\ell, m}(t) S_{\ell, m}(y)
= \sum_{\ell=0}^\infty \sqrt{A_\ell} \sum_{m=1}^{h(\ell,d)} \gb^{\ell, m}(t) S_{\ell, m}(y),
\end{equation}
where $(\gb^{\ell, m}, \ell \in \N_0, m=1,\ldots, h(\ell,d))$ 
is a sequence of independent, real-valued Brownian motions.

We next recall that the Laplace--Beltrami operator~$\Delta_{\IS^{d-1}}$ on~$\IS^{d-1}$ has the spherical harmonics $(S_{\ell, m}, \ell \in \N_0, m=1,\ldots,h(\ell,d))$ as eigenbasis with eigenvalues given by
\begin{equation*}
\Delta_{\IS^{d-1}} S_{\ell, m}
= - \ell (\ell+d-2)S_{\ell, m}
\end{equation*}
for $\ell \in \N_0$ and $m=1,\ldots,h(\ell,d)$ (see, e.\,g., \cite[Sec.~3.3]{AH12}).

We introduce Sobolev spaces on~$\IS^{d-1}$, similarly to~$\IS^2$, which are given for a smoothness index $s\in\R$ by
$$
H^s(\IS^{d-1})=(\text{Id}-\Delta_{\IS^{d-1}})^{-s/2}L^2(\IS^{d-1})
$$
together with the norm
$$
\|f\|_{H^s(\IS^{d-1})}=\|(\text{Id}-\Delta_{\IS^{d-1}})^{s/2}f\|_{L^2(\IS^{d-1})}
$$
for some $f\in H^s(\IS^{d-1})$. We also denote $H^0(\IS^{d-1}) = L^2(\IS^{d-1})$.
% https://arxiv.org/abs/1907.01571

The \emph{stochastic wave equation on~$\IS^{d-1}$} is defined as 
\begin{equation}\label{eq:swe_Sd}
\partial_{tt}u(t)-\Delta_{\IS^{d-1}}u(t)=\dot{W}(t),
\end{equation}
with initial conditions $u(0)=v_1\in L^2(\Omega;L^2(\IS^{d-1}))$ and $\partial_{t}u(0)=v_2\in L^2(\Omega;L^2(\IS^{d-1}))$, 
where $t \in \IT = [0,T]$, $T < + \infty$. The notation $\dot{W}$ stands for the formal 
derivative of the $Q$-Wiener process.

Denoting as before the velocity of the solution by $u_2 = \partial_{t} u_1 = \partial_{t} u$, 
one can rewrite \eqref{eq:swe_Sd} as 
\begin{align}\label{eq:swe2_Sd}
\diff X(t)&=AX(t)\,\diff t+G\,\diff W(t)\nonumber\\
X(0)&=X_0,
\end{align}
where 
\begin{equation*}
A=\begin{pmatrix}0 & I \\ \Delta_{\IS^{d-1}} & 0 \end{pmatrix}, 
\quad G=\begin{pmatrix}0\\I \end{pmatrix}, 
\quad X=\begin{pmatrix} u_1\\u_2\end{pmatrix}, 
\quad X_0=\begin{pmatrix} v_1\\v_2\end{pmatrix}.  
\end{equation*}
Existence of a unique mild solution follows as before.

Using the same ansatz as in Section~\ref{sec:sweS} with respect to the spherical harmonics on~$\IS^{d-1}$
\begin{align}\label{eq:ansatz_Sd}
u_1(t)=\displaystyle\sum_{\ell=0}^\infty\sum_{m=1}^{h(\ell,d)} u_1^{\ell,m}(t)S_{\ell, m}\quad\text{and}\quad 
u_2(t)=\displaystyle\sum_{\ell=0}^\infty\sum_{m=1}^{h(\ell,d)} u_2^{\ell,m}(t)S_{\ell, m}
\end{align}
we obtain
\begin{align*}
u_1^{\ell,m}(t)&=v_1^{\ell,m}+\int_0^t u_2^{\ell,m}(s)\,\diff s\\
u_2^{\ell,m}(t)&=v_2^{\ell,m}-\ell(\ell+d-2)\int_0^t u_1^{\ell,m}(s)\,\diff s+a^{\ell,m}(t),
\end{align*}
where $v_1^{\ell,m}$, $v_2^{\ell,m}$, resp.\ $a^{\ell,m}$  
are the coefficients of the expansions of the initial values $v_1$ and $v_2$, resp.\ weighted Brownian motion in the expansion of the noise~\eqref{eqW_Sd}. 

Similarly to~\eqref{eq:sys}, the variation of constants formula yields
\begin{align*}%\label{eq:sys_Sd}
\begin{cases}
u_1^{\ell,m}(t)
	&\displaystyle= R^\ell_2(t)v_1^{\ell,m}
					+R^\ell_1(t)v_2^{\ell,m}
					+\hat W_1^{\ell,m}(t)\\
u_2^{\ell,m}(t)
	&\displaystyle=-(\ell(\ell+d-2))R^\ell_1(t) v_1^{\ell,m}
					+R^\ell_2(t)v_2^{\ell,m}
					+\hat W_2^{\ell,m}(t),
\end{cases}
\end{align*}
where 
\begin{equation*}
\hat W^{\ell,m}(t)
=
\begin{pmatrix}
\hat W_1^{\ell,m}(t)\\
\hat W_2^{\ell,m}(t)
\end{pmatrix}
=\displaystyle\int_0^t R^\ell(t-s) \,\diff a^{\ell,m}(s)
\end{equation*}
with
\begin{equation*}
R^\ell (t)
= \begin{pmatrix}
R^\ell_1(t) \\ R^\ell_2(t)
\end{pmatrix}
= \begin{pmatrix}
(\ell(\ell+d-2))^{-1/2}\sin(t(\ell(\ell+d-2))^{1/2})\\
\cos(t(\ell(\ell+d-2))^{1/2})
\end{pmatrix}
\end{equation*}
for $\ell \neq 0$ and
\begin{equation*}
\hat W^{0,0}(t)
=
\begin{pmatrix}
\hat W_1^{0,0}(t)\\
\hat W_2^{0,0}(t)
\end{pmatrix}
= \begin{pmatrix}
\displaystyle\int_0^t a^{0,0}(s) \, \diff s\\
a^{0,0}(t)
\end{pmatrix}.
\end{equation*}
Note that the only change compared to Section~\ref{sec:sweS} is the value of the coefficients given by the eigenvalues~$\Delta_{\IS^{d-1}}$ and the renaming of the spherical harmonics.

As in Section~\ref{sec:conv}, we approximate the solution to the stochastic wave equation~\eqref{eq:swe2_Sd} 
by truncation of the series expansion at some finite index $\gk > 0$ and obtain
\begin{align}\label{eq:ansatzK_Sd}
u_1^\kappa(t_j)=\displaystyle\sum_{\ell=0}^\gk\sum_{m=1}^{h(\ell,d)} u_1^{\ell,m}(t_j)S_{\ell, m}\quad\text{and}\quad 
u_2^\kappa(t_j)=\displaystyle\sum_{\ell=0}^\gk\sum_{m=1}^{h(\ell,d)} u_2^{\ell,m}(t_j)S_{\ell, m}.
\end{align}

Then replacing the eigenvalues $-\ell(\ell+1)$ with $-\ell(\ell+d-2)$, the multiplicity of the eigenvalues $2\ell + 1$ with $h(\ell,d)$ and applying~\eqref{eq:conv_Lp_RF_Sd}  and~\eqref{eq:conv_as_RF_Sd} instead of Theorem~\ref{thm:iGRF_Lp_conv} in the proof of Proposition~\ref{prop:error} yields directly the following extension of Proposition~\ref{prop:error}.

\begin{proposition}%\label{prop:error_Sd} 
	Let $t \in \IT$ and $0=t_0 < \cdots < t_n = t$ be a discrete time partition 
	for $n \in \N$, which yields a recursive representation of the solution~$X=(u_1,u_2)$ of 
	the stochastic wave equation~\eqref{eq:swe2_Sd} on~$\IS^{d-1}$ given by \eqref{eq:ansatz_Sd}. 
	Assume that the initial values satisfy $v_1\in H^\beta(\IS^{d-1})$ and $v_2\in H^\gamma(\IS^{d-1})$. 
	Furthermore, assume that there exist $\ell_0 \in \N$, $\ga > 2$, and a constant~$C>0$ 
	such that the angular power spectrum of the driving noise $(A_\ell, \ell \in \N_0)$ satisfies 
	$A_\ell \le C \cdot \ell^{-\ga}$ for all $\ell > \ell_0$. 
	Then, the error of the approximate solution $X^\gk=(u_1^\gk,u_2^\gk)$, given by \eqref{eq:ansatzK_Sd}, 
	is bounded uniformly on any finite time interval and independently of the time discretization by
	\begin{align*}
	\|u_1(t) - u_1^\gk(t)\|_{L^p(\gO;L^2(\IS^{d-1}))} 
		&\le \hat{C}_p \cdot \left( \gk^{-(\ga+3-d)/2}+\gk^{-\beta}\|v_1\|_{H^\beta(\IS^{d-1})}
			+ \gk^{-(\gamma+1)}\|v_2\|_{H^\gamma(\IS^{d-1})} \right)\\
	\|u_2(t) - u_2^\gk(t)\|_{L^p(\gO;L^2(\IS^{d-1}))} 
		&\le \hat{C}_p \cdot \left( \gk^{-(\ga+1-d)/2}+ \gk^{-(\beta-1)}\|v_1\|_{H^\beta(\IS^{d-1})}
			+ \gk^{-\gamma}\|v_2\|_{H^\gamma(\IS^{d-1})} \right)
	\end{align*}
	for all $p\geq1$ and $\gk > \ell_0$, where $\hat{C}_p$ is a constant that may depend on 
	$p$, $C$, $T$, and $\ga$.
	
	Additionally, the error is bounded uniformly in time, 
	independently of the time discretization, and asymptotically in~$\gk$ by
	\begin{align*}
	\|u_1(t) - u_1^\gk(t)\|_{L^2(\IS^{d-1})} &\le \gk^{-\delta}, \quad \IP\text{-a.s.} \\
	\|u_2(t) - u_2^\gk(t)\|_{L^2(\IS^{d-1})} &\le \gk^{-(\delta-1)}, \quad \IP\text{-a.s.} 
	\end{align*}
	for all $\delta < \min((\ga+3-d)/2,\beta,\gamma+1)$.
\end{proposition}

\subsection{The free stochastic Schr\"odinger equation on~$\IS^2$}
We consider efficient simulations of paths of solutions to the free stochastic Schr\"odinger equation 
on the sphere $\IS^2$
\begin{equation}\label{eq:sch}
\ii\partial_tu(t)=\Delta_{\IS^2}u(t)+\dot W(t),
\end{equation}
with initial condition (possibly complex-valued) $u(0)\in L^2(\Omega;L^2(\IS^2))$. 
Here, the unknown $u(t)=u_R(t)+\ii u_I(t)$, 
with $t\in[0,T]$ for some $T<+\infty$, 
is a complex valued stochastic process. Furthermore, the notation $\dot W$ stands for the formal 
derivative of the (real-valued) $Q$-Wiener process with series expansion \eqref{eqW}. 

Considering the real and imaginary parts of the above SPDE, one can rewrite \eqref{eq:sch} as 
\begin{align}\label{eq:sch2}
\begin{split}
\diff X(t)&=AX(t)\,\diff t+G\,\diff W(t)\\
X(0)&=X_0,
\end{split}
\end{align}
where 
\begin{equation*}
 A=\begin{pmatrix}0 & \Delta_{\IS^2} \\ -\Delta_{\IS^2} & 0 \end{pmatrix}, 
 \quad G=\begin{pmatrix}0\\-I \end{pmatrix}, 
 \quad X=\begin{pmatrix} u_R\\u_I\end{pmatrix}, \quad X_0=\begin{pmatrix} u_R(0)\\u_I(0)\end{pmatrix}.
\end{equation*}
The existence of a mild form of the abstract formulation \eqref{eq:sch2} of 
the stochastic Schr\"odinger equation on the sphere follows like for the above stochastic wave equation. 
The mild form reads 
\begin{align}\label{eq:sch-mild}
X(t)=\e^{tA}X_0+\int_0^t\e^{(t-s)A}G\,\diff W(s)
\end{align}
with the semigroup 
$$
\e^{tA}=
\begin{pmatrix} 
\cos(t\Delta_{\IS^2}) & \sin(t\Delta_{\IS^2}) \\
-\sin(t\Delta_{\IS^2}) & \cos(t\Delta_{\IS^2})
\end{pmatrix}.
$$
Finally, one obtains the integral formulation of the above problem as 
\begin{align*}%\label{eq:sch-int}
\begin{cases}
u_R(t)&\displaystyle=u_R(0)+\int_0^t\Delta_{\IS^2}u_I(s)\,\diff s\\
u_I(t)&\displaystyle=u_I(0)-\int_0^t\Delta_{\IS^2}u_R(s)\,\diff s-W(t).
\end{cases}
\end{align*}

As it was done for the stochastic wave equation in Section~\ref{sec:sweS}, one can make the following 
ansatz for the real and imaginary part of solutions to~\eqref{eq:sch-mild}
$$
u_R(t)=\sum_{\ell=0}^\infty\sum_{m=-\ell}^\ell u_R^{\ell,m}Y_{\ell,m}\quad\text{and}\quad
u_I(t)=\sum_{\ell=0}^\infty\sum_{m=-\ell}^\ell u_I^{\ell,m}Y_{\ell,m}
$$
and find the following system of equations defining the coefficients of these expansions:
\begin{align*}%\label{eq:sch-sys}
\begin{cases}
u_R^{\ell,m}(t)&\displaystyle=\cos(t(\ell(\ell+1))^{1/2})v_R^{\ell,m}+\sin(t(\ell(\ell+1))^{1/2})v_I^{\ell,m}+\hat W_R^{\ell,m}(t)\\
u_I^{\ell,m}(t)&\displaystyle=-\sin(t(\ell(\ell+1))^{1/2})v_R^{\ell,m}+\cos(t(\ell(\ell+1))^{1/2})v_I^{\ell,m}+\hat W_I^{\ell,m}(t),
\end{cases}
\end{align*}
where 
\begin{equation*}
\hat W^{\ell,m}(t)
=
\begin{pmatrix}
\hat W_R^{\ell,m}(t)\\
\hat W_I^{\ell,m}(t)
\end{pmatrix}
=
\begin{pmatrix}
\displaystyle\int_0^t \sin((t-s)(\ell(\ell+1))^{1/2}) \,\diff a^{\ell,m}(s)\\
\displaystyle\int_0^t \cos((t-s)(\ell(\ell+1))^{1/2}) \,\diff a^{\ell,m}(s)
\end{pmatrix}
\end{equation*}
and $v_R^{\ell,m}$, resp. $v_I^{\ell,m}$, are the coefficients of the real, resp.\ imaginary, part of the initial value $u(0)$.

It is clear that the analysis from Section~\ref{sec:conv} can directly be extend to the case of the stochastic Schr\"odinger equation on the sphere \eqref{eq:sch}. The errors in the truncation procedure, denoted by $u_R^\gk$ and $u_I^\gk$, 
of the above ansatz are given by the following proposition (presented for zero initial data for simplicity).
\begin{proposition}\label{prop:error-sch} 
Let $t \in \IT=[0,T]$ and $0=t_0 < \ldots < t_n = t$ be a discrete time partition 
for $n \in \N$, which yields a recursive representation of the solution~$X=(u_R,u_I)$ of 
the stochastic Schr\"odinger equation on the sphere~\eqref{eq:sch2} with initial data $u(0)=0$. 
Assume that there exist $\ell_0 \in \N$, $\ga > 2$, and a constant~$C>0$ 
such that the angular power spectrum of the driving noise $(A_\ell, \ell \in \N_0)$ decays with 
$A_\ell \le C \cdot \ell^{-\ga}$ for all $\ell > \ell_0$. 
Then, the error of the approximate solution $X^\gk=(u_R^\gk,u_I^\gk)$ is bounded uniformly 
in time and independently of the time discretization by
\begin{align*}
\|u_R(t) - u_R^\gk(t)\|_{L^p(\gO;L^2(\IS^2))} &\le \hat{C}_p \cdot \gk^{-(\ga/2-1)}\\
\|u_I(t) - u_I^\gk(t)\|_{L^p(\gO;L^2(\IS^2))} &\le \hat{C}_p \cdot \gk^{-(\ga/2-1)}
\end{align*}
for all $p\geq1$ and $\gk > \ell_0$, where $\hat{C}_p$ is a constant that may depend on 
$p$, $C$, $T$, and $\ga$.

On top of that, the error is bounded uniformly in time, 
independently of the time discretization, and asymptotically in~$\gk$ by
\begin{align*}
\|u_R(t) - u_R^\gk(t)\|_{L^2(\IS^2)} &\le \gk^{-(\delta-1)}, \quad \IP\text{-a.s.} \\
\|u_I(t) - u_I^\gk(t)\|_{L^2(\IS^2)} &\le \gk^{-(\delta-1)}, \quad \IP\text{-a.s.} 
\end{align*}
for all $\delta < \ga/2$.
\end{proposition}
Since the proof of this proposition follows the lines of the proof of Proposition~\ref{prop:error}, we omit it and 
instead present some numerical experiments illustrating these theoretical results. 

We compute the errors when approximating solutions to~\eqref{eq:sch} 
for various truncation indices~$\gk$ for a $Q$-Wiener process with parameter $\ga=4$. 
All other parameters are the same as in Section~\ref{sec:num}. In Figure~\ref{fig:sch}, 
we display a sample at time $T=1$ and a strong convergence plot of the real part of the numerical approximation, 
as well as errors in the approximation of a path (imaginary part) 
to the stochastic Schr\"odinger equation on the sphere. 
These illustrations are in agreement with 
the results from Proposition~\ref{prop:error-sch}.

\begin{figure}
\subfigure[Sample of solution.]{\includegraphics[width=0.32\textwidth]{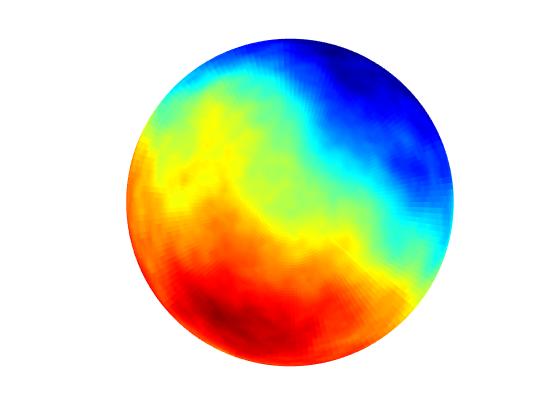}}
\subfigure[Strong errors (real part).]{\includegraphics[width=0.32\textwidth]{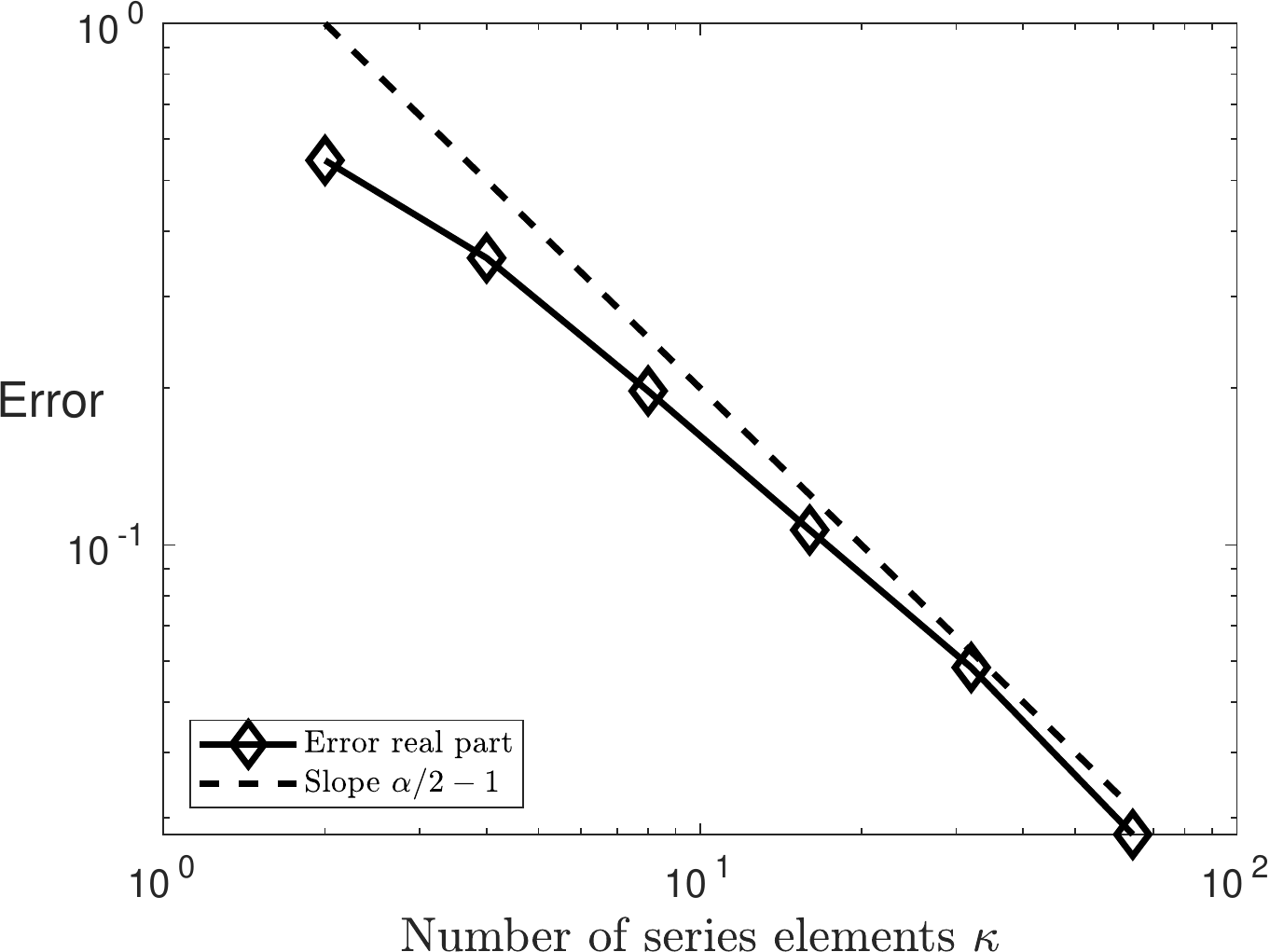}}
\subfigure[Error in approximation of one path (imaginary part).]{\includegraphics[width=0.32\textwidth]{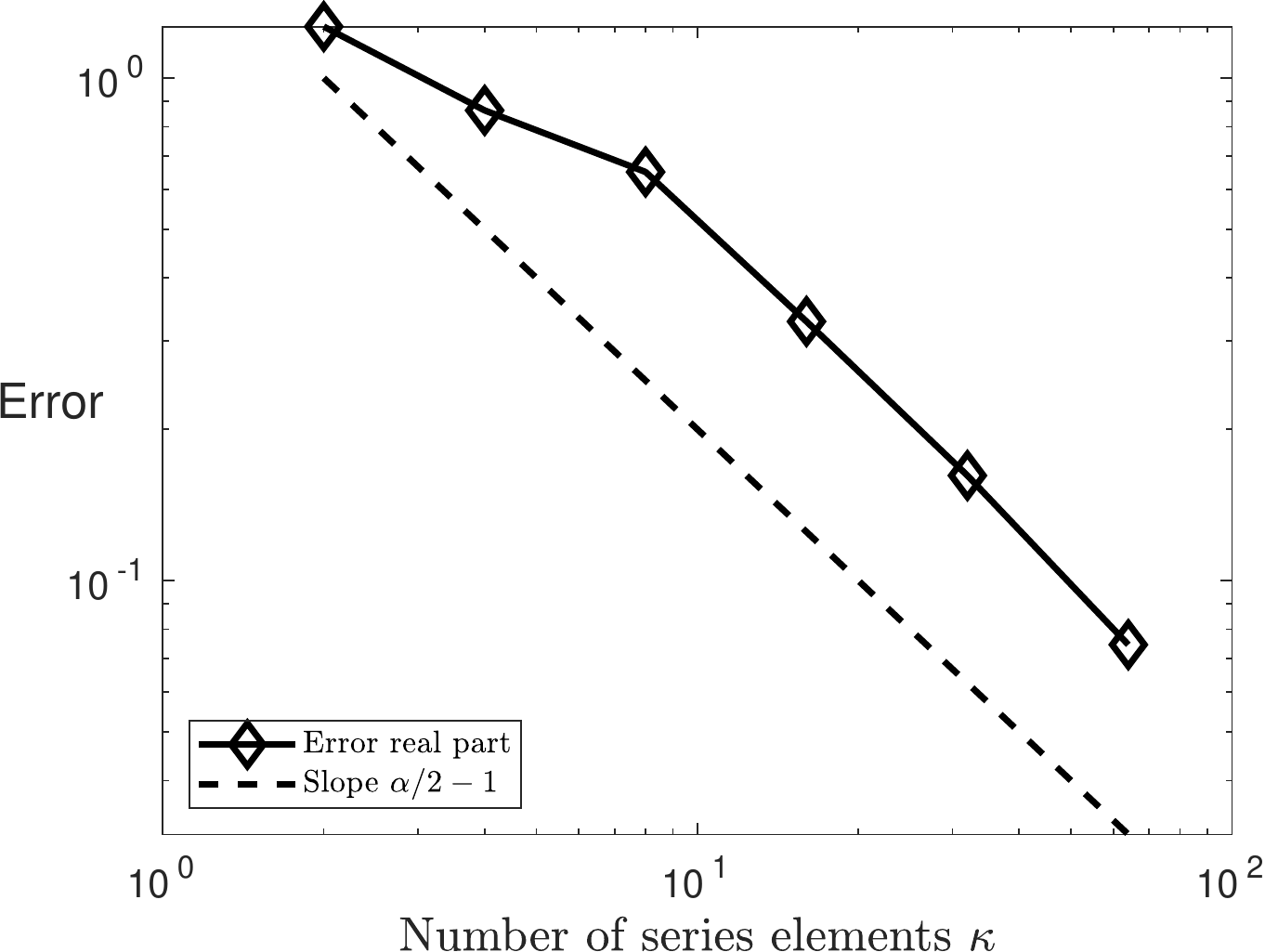}}
\caption{Sample, mean-square errors, and error of one path of the approximation of the stochastic Schr\"odinger equation 
with angular power spectrum of the $Q$-Wiener process with parameter $\ga=4$ 
and $100$~Monte Carlo samples (for the mean-square errors).}
\label{fig:sch}
\end{figure}

\bibliographystyle{plain}
\bibliography{swe_sphere}

\begin{thebibliography}{10}

\bibitem{AKL16}
Adam Andersson, Raphael Kruse, and Stig Larsson.
\newblock Duality in refined {S}obolev--{M}alliavin spaces and weak
  approximations of {SPDE}.
\newblock {\em Stoch. PDE: Anal. Comp.}, 4(1):113--149, 2016.

\bibitem{123456}
Vo~V. Anh, Philip Broadbridge, Andriy Olenko, and Yu~Guang Wang.
\newblock On approximation for fractional stochastic partial differential
  equations on the sphere.
\newblock {\em Stoch. Environ. Res. Risk Assess}, 32(9):2585--2603, 2018.

\bibitem{MR3484400}
Rikard Anton, David Cohen, Stig Larsson, and Xiaojie Wang.
\newblock Full discretization of semilinear stochastic wave equations driven by
  multiplicative noise.
\newblock {\em SIAM J. Numer. Anal.}, 54(2):1093--1119, 2016.

\bibitem{AH12}
Kendall Atkinson and Weimin Han.
\newblock {\em Spherical Harmonics and Approximations on the Unit Sphere: An
  Introduction}, volume 2044 of {\em Lecture Notes in Mathematics}.
\newblock Springer-Verlag, 2012.

\bibitem{BHS18}
Charles-Edouard Br\'ehier, Martin Hairer, and Andrew~M. Stuart.
\newblock Weak error estimates for trajectories of {SPDEs} under spectral
  {G}alerkin discretization.
\newblock {\em J. Comp. Math.}, 36(2):159--182, 2018.

\bibitem{MR4031900}
Phil Broadbridge, Alexander~D. Kolesnik, Nikolai Leonenko, and Andriy Olenko.
\newblock Random spherical hyperbolic diffusion.
\newblock {\em J. Stat. Phys.}, 177(5):889--916, 2019.

\bibitem{C12}
Julia Charrier.
\newblock Strong and weak error estimates for elliptic partial differential
  equations with random coefficients.
\newblock {\em SIAM J. Numer. Anal.}, 50(1):216--246, 2012.

\bibitem{MR3763911}
Jorge Clarke De~la Cerda, Alfredo Alegr\'{\i}a, and Emilio Porcu.
\newblock Regularity properties and simulations of {G}aussian random fields on
  the sphere cross time.
\newblock {\em Electron. J. Stat.}, 12(1):399--426, 2018.

\bibitem{10.3389/fphys.2018.01052}
Richard~H. Clayton.
\newblock Dispersion of recovery and vulnerability to re-entry in a model of
  human atrial tissue with simulated diffuse and focal patterns of fibrosis.
\newblock {\em Frontiers in Physiology}, 9:1052, 2018.

\bibitem{MR3033008}
David Cohen, Stig Larsson, and Magdalena Sigg.
\newblock A trigonometric method for the linear stochastic wave equation.
\newblock {\em SIAM J. Numer. Anal.}, 51(1):204--222, 2013.

\bibitem{DPZ92}
Giuseppe Da~Prato and Jerzy Zabczyk.
\newblock {\em Stochastic Equations in Infinite Dimensions}, volume 152 of {\em
  Encyclopedia of Mathematics and its Applications}.
\newblock Cambridge University Press, second edition, 2014.

\bibitem{Dalang2009}
Robert Dalang, Davar Khoshnevisan, Carl Mueller, David Nualart, and Yimin Xiao.
\newblock {\em A Minicourse on Stochastic Partial Differential Equations},
  volume 1962 of {\em Lecture Notes in Mathematics}.
\newblock Springer-Verlag, 2009.
\newblock Held at the University of Utah, Salt Lake City, UT, May 8--19, 2006,
  Edited by Khoshnevisan and Firas Rassoul-Agha.

\bibitem{DP09}
Arnaud Debussche and Jacques Printems.
\newblock Weak order for the discretization of the stochastic heat equation.
\newblock {\em Math. Comput.}, 78(266):845--863, 2009.

\bibitem{F70}
Xavier Fernique.
\newblock Int\'egrabilit\'e des vecteurs gaussiens.
\newblock {\em C. R. Acad. Sci., Paris, S\'er. A}, 270:1698--1699, 1970.

\bibitem{1002247}
Joshua~H. Goldwyn and Eric Shea-Brown.
\newblock The what and where of adding channel noise to the {H}odgkin--{H}uxley
  equations.
\newblock {\em PLOS Comp. Biol.}, 7(11):1--9, 11 2011.

\bibitem{HM19}
Philipp Harms and Marvin~S. M\"{u}ller.
\newblock Weak convergence rates for stochastic evolution equations and
  applications to nonlinear stochastic wave, {HJMM}, stochastic
  {S}chr\"{o}dinger and linearized stochastic {K}orteweg--de {V}ries equations.
\newblock {\em Z. Angew. Math. Phys.}, 70(1):Paper No. 16, 28, 2019.

\bibitem{Hasselmann}
Klaus Hasselmann.
\newblock Stochastic climate models part {I.} {T}heory.
\newblock {\em Tellus}, 28(6):473--485, 1976.

\bibitem{MR4059369}
Lukas Herrmann, Kristin Kirchner, and Christoph Schwab.
\newblock Multilevel approximation of {G}aussian random fields: fast
  simulation.
\newblock {\em Math. Models Methods Appl. Sci.}, 30(1):181--223, 2020.

\bibitem{MR3768993}
Lukas Herrmann, Annika Lang, and Christoph Schwab.
\newblock Numerical analysis of lognormal diffusions on the sphere.
\newblock {\em Stoch. PDE: Anal. Comp.}, 6(1):1--44, 2018.

\bibitem{JJW18}
Ladislas Jacobe~de Naurois, Arnulf Jentzen, and Timo Welti.
\newblock Lower bounds for weak approximation errors for spatial spectral
  galerkin approximations of stochastic wave equations.
\newblock In Andreas Eberle, Martin Grothaus, Walter Hoh, Moritz Kassmann,
  Wilhelm Stannat, and Gerald Trutnau, editors, {\em Stochastic Partial
  Differential Equations and Related Fields}, pages 237--248, Cham, 2018.
  Springer International Publishing.

\bibitem{PhysRevLett.56.889}
Mehran Kardar, Giorgio Parisi, and Yi-Cheng Zhang.
\newblock Dynamic scaling of growing interfaces.
\newblock {\em Phys. Rev. Lett.}, 56:889--892, 1986.

\bibitem{MR3907363}
Yoshihito Kazashi and Quoc~T. Le~Gia.
\newblock A non-uniform discretization of stochastic heat equations with
  multiplicative noise on the unit sphere.
\newblock {\em J. Complexity}, 50:43--65, 2019.

\bibitem{KLP20}
Mih\'aly Kov\'acs, Annika Lang, and Andreas Petersson.
\newblock Weak convergence of fully discrete finite element approximations of
  semilinear hyperbolic {SPDE}s with additive noise.
\newblock {\em ESAIM:M2AN}, 54(6):2199--2227, 2020.

\bibitem{KLL12}
Mih\'aly Kov\'acs, Stig Larsson, and Fredrik Lindgren.
\newblock Weak convergence of finite element approximations of linear
  stochastic evolution equations with additive noise.
\newblock {\em BIT Num. Math}, 52(1):85--108, 2012.

\bibitem{MR1874654}
Yuriy~V. Kozachenko and L.~F. Kozachenko.
\newblock Modeling {G}aussian isotropic random fields on a sphere.
\newblock {\em J. Math. Sci.}, 107(2):3751--3757, 2001.

\bibitem{MR2568294}
Xiaohong Lan and Domenico Marinucci.
\newblock On the dependence structure of wavelet coefficients for spherical
  random fields.
\newblock {\em Stochastic Process. Appl.}, 119(10):3749--3766, 2009.

\bibitem{MR3769662}
Xiaohong Lan, Domenico Marinucci, and Yimin Xiao.
\newblock Strong local nondeterminism and exact modulus of continuity for
  spherical {G}aussian fields.
\newblock {\em Stochastic Process. Appl.}, 128(4):1294--1315, 2018.

\bibitem{LLS13}
Annika Lang, Stig Larsson, and {Ch}ristoph Schwab.
\newblock Covariance structure of parabolic stochastic partial differential
  equations.
\newblock {\em Stoch. PDE: Anal. Comp.}, 1(2):351--364, 2013.

\bibitem{MR3404631}
Annika Lang and Christoph Schwab.
\newblock Isotropic {G}aussian random fields on the sphere: regularity, fast
  simulation and stochastic partial differential equations.
\newblock {\em Ann. Appl. Probab.}, 25(6):3047--3094, 2015.

\bibitem{MR4091198}
Quoc~Thong Le~Gia, Ian~H. Sloan, Robert~S. Womersley, and Yu~Guang Wang.
\newblock Isotropic sparse regularization for spherical harmonic
  representations of random fields on the sphere.
\newblock {\em Appl. Comput. Harmon. Anal.}, 49(1):257--278, 2020.

\bibitem{MP11}
Domenico Marinucci and Giovanni Peccati.
\newblock {\em Random Fields on the Sphere. Representation, Limit Theorems and
  Cosmological Applications}.
\newblock Cambridge University Press, 2011.

\bibitem{M98}
Mitsuo Morimoto.
\newblock {\em Analytic Functionals on the Sphere}, volume 178 of {\em
  Translations of Mathematical Monographs}.
\newblock American Mathematical Society, 1998.

\bibitem{Szego}
G{\'a}bor Szeg{\H{o}}.
\newblock {\em Orthogonal Polynomials}, volume XXIII of {\em Colloquium
  Publications}.
\newblock American Mathematical Society, fourth edition, 1975.

\bibitem{W15}
Xiaojie Wang.
\newblock An exponential integrator scheme for time discretization of nonlinear
  stochastic wave equation.
\newblock {\em J. Sci. Comput.}, 64(1):234--263, 2015.

\bibitem{Y83}
Myhailo~I. Yadrenko.
\newblock {\em Spectral Theory of Random Fields}.
\newblock Translation Series in Mathematics and Engineering. Optimization
  Software, Inc., Publications Division; Springer-Verlag, 1983.
\newblock Transl. from the Russian.

\end{thebibliography}

\end{document}